\documentclass[11pt]{article}

\usepackage[a4paper,margin=1in]{geometry}
\usepackage{amsmath,amssymb,amsfonts,amsthm,mathtools}
\usepackage{enumitem}
\usepackage{xcolor}
\usepackage{hyperref}
\usepackage{microtype}
\usepackage{bm}
\usepackage{tikz}
\usetikzlibrary{arrows.meta,calc,decorations.pathreplacing,patterns}

\hypersetup{colorlinks=true,linkcolor=black,citecolor=black,urlcolor=black}
\numberwithin{equation}{section}

\newtheorem{theorem}{Theorem}[section]
\newtheorem{lemma}[theorem]{Lemma}
\newtheorem{proposition}[theorem]{Proposition}
\newtheorem{corollary}[theorem]{Corollary}

\newcommand{\ii}{\mathbf i}

\title{Convergence of the PML-BIE Method for Acoustic Scattering in an Impedance Half-Space}
\author{Wangtao Lu\thanks{School of Mathematical Sciences, Zhejiang University, Hangzhou, China (\texttt{wangtaolu@zju.edu.cn}).}}
\date{}

\begin{document}
\maketitle

\begin{abstract}
The perfectly matched layer-based boundary integral equation (PML-BIE) method (Lu et al., \emph{SIAM J. Appl. Math.} 78 (2018)) has become an effective tool for wave scattering problems in unbounded domains. Despite its successful applications, a rigorous convergence theory has remained incomplete in many physically relevant settings.  In this paper, we present a general framework for establishing convergence of PML-BIE methods, using acoustic scattering in an impedance half-space as the first illustrative example. The framework separates the analysis into three components: convergence of the PML truncated partial differential equation, equivalence between the PML problem and an exact PML-Green BIE, and convergence of a computable PML-BIE obtained by replacing the exact kernel with a stretched free-space kernel.  For the impedance half-space problem, we prove convergence of both the exact and computable PML-BIE formulations.  The final error bound decays exponentially as the PML absorption power increases.  The resulting theory provides a rigorous foundation for the PML-BIE method and gives a reference model for more general scattering problems in layered and inhomogeneous media.
\end{abstract}



\section{Introduction}

Boundary integral equation (BIE) methods are attractive for time-harmonic
scattering because they reduce the dimension of the unknown and encode radiation
through Green functions.  Their effectiveness, however, depends strongly on the
choice of the Green kernel.  For a flat Dirichlet or Neumann half-plane, the
background Green function is given by one image.  For impedance, layered, or
transmission backgrounds, the corresponding Green functions are substantially
more complicated.  The perfectly matched layer (PML), introduced by B\'erenger for
electromagnetic waves \cite{Berenger1994}, provides a powerful and flexible way
to accurately truncate unbounded wave propagation problems with exponentially
small errors.  The PML idea has since become a standard tool in computational
wave scattering, including both finite element and boundary integral
formulations.

The PML-BIE method defines boundary integral operators using the PML-transformed
free-space Green function.  The rationale is the following.  The true background
Green function yields a localized boundary integral equation, but the kernel may
be expensive or difficult to evaluate.  The free-space Green function is easier
to compute, but the associated boundary integral equation is generally posed on
unbounded interfaces and is therefore not naturally localized.
The PML-BIE replaces the free-space Green function
$\Phi_k(x,y)=\frac{\ii}{4}H_0^{(1)}(k|x-y|)$ by the stretched Green function
\begin{equation*}
  \Phi_\sigma(x,y)=\Phi_k(F_\sigma(x),F_\sigma(y)),
\end{equation*}
where $F_\sigma$ is the complex PML stretching map.  This kernel is simple to
evaluate and is exponentially damped in the PML region.  Since the
layered-medium PML-BIE method of \cite{LuLuQian2018}, related PML-BIE solvers
have been developed for layered orthotropic media \cite{GaoLu2022}, locally
defected periodic surfaces \cite{YuHuLuRathsfeld2022}, step-like interfaces
\cite{Lu2021,LuLi2025}, and three-dimensional electromagnetic/acoustic
layered-medium scattering \cite{LuXuYinZhang2023,BaoLuYinZhang2024,WangLu2026}.
Besides, related complex-scaling integral-equation
methods have also been used for time-harmonic water waves
\cite{BonnetBenDhiaFariaPerez2024}, open waveguides and unbounded interfaces
\cite{EpsteinGoodwillHoskinsQuinnRachh2025,GoodwillEpstein2026}, eigenvalues
problems for the Neumann--Poincar\'e operator on corner domains
\cite{FariaMonteghetti2025}, and fast multipole acceleration with complex
coordinates \cite{GoodwillGreengardHoskinsRachhWang2025}. All of the above
works show that complex-stretched kernels can replace more complicated physical
Green functions in open-domain problems.  A remaining question is to identify a
convergence mechanism explaining why a finite PML-BIE inherits the exponential
convergence of the underlying PML truncation.

The PDE theory of PML truncations is much better developed.  Early adaptive and
convergence analyses for acoustic and periodic scattering include
\cite{ChenWu2003,ChenLiu2005}; anisotropic and electromagnetic PML methods were
analyzed in, for example, \cite{ChenCuiZhang2013}.  For layered backgrounds,
stability and exponential convergence of uniaxial PMLs were proved for acoustic
scattering in \cite{ChenZheng2010} and for electromagnetic scattering in
\cite{ChenZheng2017}.  The well-posedness and exponential convergence of a UPML
method for acoustic scattering by a locally perturbed impedance line were proved
in \cite{JiangLi2020}.  Related radiation-condition and PML convergence theories
for rough, periodic, and step-like surfaces were developed in
\cite{ChandlerWildeMonk2009,Lu2021,Zhang2022,LuZhengZhu2025}.  The
modified-Green-function and Hankel-function estimates developed in
\cite{ChenLiu2005,ChenZheng2010,JiangLi2020,LuLaiWu2024},
are also the estimates used below to compare exact PML Green functions with
stretched free-space kernels.  These PDE results estimate the solution of a
truncated PML boundary value problem; an additional boundary-integral argument
is needed to transfer such estimates to PML-BIEs.

A different related direction is the complexified BIE analysis in
\cite{EpsteinGreengardHoskinsJiangRachh2024} for the Helmholtz equation with
Dirichlet boundary conditions in locally perturbed half-spaces.  The analysis
treats analytic continuation of a Dirichlet double-layer density and contour
deformation of the resulting integral.  In the Dirichlet half-space setting, the
flat background Green function is already available by the elementary image
construction, so the motivation differs from the impedance and layered settings
considered here.  

This paper establishes a convergence theory for PML-BIE methods and formulates a
proof mechanism that can be adapted to more general settings.  The half-space
acoustic scattering problem with an impedance boundary condition is used as the
first model problem because its operator structure is transparent while
retaining the key obstruction: the true physical Green function is nontrivial
and contains both propagating and surface-wave contributions \cite{JiangLi2020}.
The final truncation error is proved to be exponentially small as the PML
absorption power increases. Transmission, layered, and water-wave problems
require different background Calder\'on systems, but the mechanism developed
here is intended to be reusable in those settings. 

The proof proceeds by inserting an exact finite-domain PML Green function
between the PDE-PML convergence theorem and the computable PML-BIE.  The
artificial boundary used in this paper is a homogeneous Neumann boundary rather
than the Dirichlet boundary often used in PML truncations.  This choice is tied
to the open-arc trace space of the computable PML-BIE.  The unknown trace is
sought in the ordinary space $H^{1/2}$; a Dirichlet boundary would naturally
lead to endpoint-compatible spaces such as $H^{1/2}_{00}$, whereas the Neumann
boundary keeps the retained impedance trace in the same open-arc space as the
computable BIE.

We next summarize the proof framework.  Section~\ref{sec:problem} defines the
impedance scattering problem, the
retained arc $\Gamma_R$, the open-arc trace space
\[
  Y_R=H^{1/2}(\Gamma_R),
  \qquad Y_R'=H^{-1/2}_{00}(\Gamma_R),
\]
and the Neumann PML truncation.  Section~\ref{sec:conv-pmlpde} proves the
Neumann PML-PDE estimate, including the trace stability and PML-segment
decay needed later.  Section~\ref{sec:exact-green-pde} constructs the exact
Neumann PML Green function $G_{\sigma,N}^0$, derives the exact PML-Green BIE
\begin{equation*}
  \mathcal A_{\sigma,N}^0\phi_\sigma=\mathcal S_{\sigma,N}^0g,
  \qquad
  \mathcal A_{\sigma,N}^0=\frac12 I+\mathcal K_{\sigma,N}^0+
  \mathcal S_{\sigma,N}^0M_\beta,
\end{equation*}
proves its equivalence with the Neumann PML PDE, and obtains the convergence and
polynomial stability.  Section~\ref{sec:computable} studies the computable
PML-BIE
\[
  \mathcal A_\sigma^c\phi_\sigma^c=\mathcal S_\sigma^c g,
  \qquad
  \mathcal A_\sigma^c=\frac12I+\mathcal K_\sigma^c+
  \mathcal S_\sigma^c M_\beta,
\]
which uses the stretched free-space Green function $\Phi_\sigma$.
We decompose the difference ${\cal A}_{\sigma,N}^0 - {\cal A}_{\sigma}^c$
into the localized nearest-image perturbation $\mathcal B_\sigma^{\rm im}$ and
an exponentially small
remainder, and prove the estimate $\|(\mathcal A_{\sigma,N}^0)^{-1} {\cal
B}_{\rm im}\|\to0$. This immediately implies the unique solvability of the
computable PML-BIE and the exponential convergence of the approximate density.
Section~\ref{sec:cs-bie} recalls the complex-scaled BIE interpretation, and
Section~\ref{sec:conclusion} concludes the paper.

\section{Problem description}\label{sec:problem}

\subsection{\texorpdfstring{The impedance scattering problem}{The impedance scattering problem}}

We follow the notation of \cite{JiangLi2020} as closely as possible. The original unbounded scattering geometry is shown in Figure~\ref{fig:geometry}(a).  Let
\begin{equation}
  \Sigma=\{(x_1,p(x_1)):x_1\in\mathbb R\},
  \qquad
  \mathbb R^2_{\Sigma,+}:=\{x\in\mathbb R^2:x_2>p(x_1)\},
\end{equation}
where the profile $p$ is piecewise $C^1$ in $[-1,1]$. Thus $\Sigma$ is a local perturbation of
\begin{equation}
  \Sigma_0:=\{(x_1,0):x_1\in\mathbb R\}.
\end{equation}
We write
\begin{equation}
  \Sigma=\Sigma_\infty\cup\Sigma_p,
  \qquad
  \Sigma_\infty:=\{(x_1,0): |x_1|\ge 1\},
  \qquad
  \Sigma_p:=\Sigma\setminus\Sigma_\infty .
\end{equation}
The unit normal $\nu$ points to the exterior of $\mathbb R^2_{\Sigma,+}$, i.e. downward on the flat part of the boundary.
\begin{figure}[t]
\centering
\begin{tikzpicture}[
  scale=0.92,
  >=Latex,
  font=\footnotesize,
  bdry/.style={very thick,blue!70!black},
  trunc/.style={very thick,red!75!black},
  outerbd/.style={thick,black},
  pml/.style={fill=gray!13},
  core/.style={fill=white},
  ray/.style={->,thick,gray!70!black}
]
\begin{scope}[xshift=-4.15cm]
  \node[font=\small\bfseries] at (0,2.55) {(a) Original scattering problem};

  \draw[densely dashed,gray!70] (-3.25,0) -- (-1.45,0);
  \draw[densely dashed,gray!70] (1.45,0) -- (3.25,0);
  \node[gray!70,anchor=north east] at (-2.35,-0.03) {$\Sigma_0$};

  \draw[bdry]
    plot[smooth,tension=0.85] coordinates
    {(-3.35,0) (-2.45,0) (-1.45,0.04) (-0.70,0.25)
     (0,0.42) (0.70,0.25) (1.45,0.04) (2.45,0) (3.35,0)};
  \node[blue!70!black,anchor=north west] at (2.25,-0.02) {$\Sigma$};

  \draw[trunc]
    plot[smooth,tension=0.85] coordinates
    {(-1.32,0.05) (-0.70,0.25) (0,0.42) (0.70,0.25) (1.32,0.05)};
  \node[red!75!black,anchor=south] at (0,0.36) {$\Sigma_p$};

  \node[blue!55!black] at (0,1.50) {$\mathbb R^2_{\Sigma,+}$};
  \draw[->,blue!70!black] (-0.27,0.42) -- (0.3,-0.7);
  \node[blue!70!black,anchor=south east] at (0.3,-0.7) {$\nu$};

\end{scope}

\begin{scope}[xshift=4.15cm]
  \node[font=\small\bfseries] at (0,2.55) {(b) Neumann PML truncation};

  \fill[pml] (-3.25,-0.95) rectangle (3.25,2.15);
  \fill[core] (-2.08,-0.32) rectangle (2.08,1.45);
  \draw[dashed,gray!75] (-2.08,-0.95) -- (-2.08,2.15);
  \draw[dashed,gray!75] ( 2.08,-0.95) -- ( 2.08,2.15);
  \draw[dashed,gray!75] (-3.25,1.45) -- (3.25,1.45);
  \node[gray!70] at (0,1.78) {PML layers};

  \draw[outerbd] (-3.25,-0.95) rectangle (3.25,2.15);
  \node[anchor=north west] at (-3.18,2.10) {$\mathcal D_R$};
  \node[anchor=south east] at (3.18,-0.92) {$\Sigma_R=\partial\mathcal D_R$};

  \draw[bdry,dashed] (-3.65,0) -- (-3.25,0);
  \draw[bdry,dashed] (3.25,0) -- (3.65,0);
  \draw[trunc]
    plot[smooth,tension=0.85] coordinates
    {(-3.25,0) (-2.45,0) (-1.45,0.05) (-0.72,0.25)
     (0,0.42) (0.72,0.25) (1.45,0.05) (2.45,0) (3.25,0)};
  \fill[red!75!black] (-3.25,0) circle (1.3pt);
  \fill[red!75!black] ( 3.25,0) circle (1.3pt);
  \node[red!75!black,anchor=south] at (0,0.32) {$\Gamma_R$};
  \node[blue!70!black,anchor=north] at (3.52,-0.02) {$\Sigma$};

  \node[blue!55!black] at (0,1.1) {$\Omega_R$};
  \node[blue!55!black] at (-0.7,-0.4) {$\Omega_R^-$};

  \draw[->,red!75!black] (-0.25,0.42) -- (0.3,-0.7);
  \node[red!75!black,anchor=south east] at (0.3,-0.7) {$\nu_\sigma$};

  \node[anchor=north west,align=left] at (-3.20,-1.02)
    {Neumann boundary: $\partial_{\nu_\sigma}u_\sigma=0$ on $\Sigma_R$};
\end{scope}
\end{tikzpicture}
\caption{Geometry of the impedance scattering problem and its Neumann PML truncation.}  
\label{fig:geometry}
\end{figure}
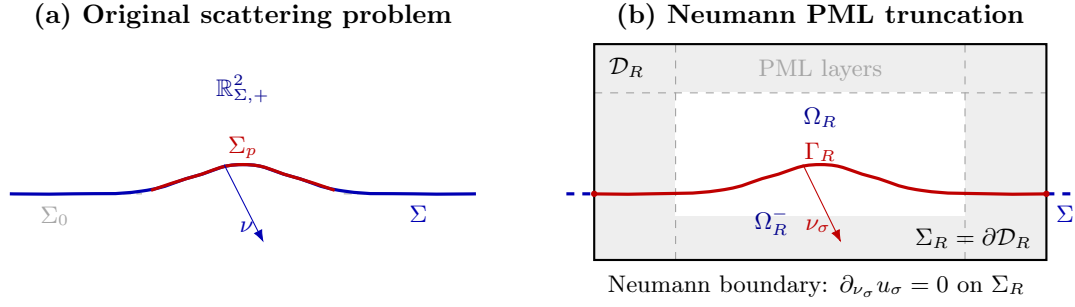
The physical boundary condition is homogeneous for the total field $U$:
\begin{align}
  \Delta U+k^2U&=0 &&\text{in }\mathbb R^2_{\Sigma,+},\label{eq:total-pde}\\
  \partial_\nu U-\ii k\beta U&=0 &&\text{on }\Sigma.\label{eq:total-imp}
\end{align}
Here $k>0$ is the wave number and $\beta\in L^\infty(\Sigma)$ is the acoustic admittance.  Following the impedance convention in \cite{JiangLi2020}, we assume
\begin{equation}\label{eq:beta-assump}
  \operatorname{Re}\beta\ge0,
  \qquad
  \beta\equiv -\ii Z/k\quad\text{on }\Sigma_\infty,
  \qquad Z>0 .
\end{equation}
With this convention, the flat part of the boundary condition is equivalent to
\begin{equation}
  \partial_\nu U-ZU=0\quad\text{on }\Sigma_\infty .
\end{equation}

Let $U^{\rm inc}$ be an incident field.  To extract an outgoing field, we subtract a known reference solution $U^{\rm ref}$ associated with the flat impedance half-plane.  For a plane wave, $U^{\rm ref}$ is the incident plus reflected flat-boundary solution.  For a point source, one may subtract either the free-space Green function $\Phi_k$, which gives boundary data that is not compactly supported but is rapidly damped in the PML, or the flat impedance background Green function \cite[Sec. 2.4]{JiangLi2020}, which gives compactly supported boundary data.  To simplify the presentation, we consider the compactly supported case only.

Define
  $u:=U-U^{\rm ref}$.
Then $u$ satisfies 
\begin{align}
  \Delta u+k^2u&=0 &&\text{in }\mathbb R^2_{\Sigma,+},\label{eq:phys-pde}\\
  \partial_\nu u-\ii k\beta u&=g &&\text{on }\Sigma,\label{eq:phys-imp}
\end{align}
where
\begin{equation}\label{eq:g-from-reference}
  g:=-(\partial_\nu U^{\rm ref}-\ii k\beta U^{\rm ref})\in H^{-1/2}_{00}(\Sigma_p)\hookrightarrow H^{-1/2}(\Sigma).
\end{equation}
At infinity, $u$ contains both the ordinary outgoing component and the surface-wave component generated by the impedance boundary, i.e.,
\begin{equation}\label{eq:phys-src}
  \lim_{r\to\infty}\int_{S_r^1}
  |\partial_r u-\ii k u|^2ds=0,
  \qquad
  \lim_{r\to\infty}\int_{S_r^2}
  |\partial_r u-\ii\sqrt{Z^2+k^2}u|^2ds=0,
\end{equation}
where 
  $S_r:=\{x\in\mathbb R^2_+:|x|=r\}$,
  $S_r^1:=\{x\in S_r:x_2\ge r^{1/4}\}$, and
  $S_r^2:=S_r\setminus S_r^1$.
The first Sommerfeld-like radiation condition selects the outgoing propagating wave, while the second
selects the outgoing surface wave along the flat boundary. We refer readers to
\cite{JiangLi2020} and the references therein for the well-posedness of the scattering problem \eqref{eq:phys-pde}--\eqref{eq:phys-src}.

\subsection{\texorpdfstring{UPML and the Neumann PML problem}{UPML and the Neumann PML problem}}
We now introduce the UPML stretching in the notation of \cite{JiangLi2020}. The complex coordinate stretching is defined as
\begin{equation}\label{eq:pml-stretch-explicit}
  \widetilde x_j=F_{\sigma,j}(x_j):=\int_0^{x_j}\alpha_j(t)\,dt,
  \qquad j=1,2,
\end{equation}
where 
\begin{equation}\label{eq:alpha-j}
  \alpha_j(t):=1+\ii\sigma_j(t),
  \qquad j=1,2,
\end{equation}
and
\begin{equation}\label{eq:sigma-profile-JL}
  \sigma_j(t)\ge0,
  \qquad
  \sigma_j(t)=\sigma_j(-t),
  \qquad
  \sigma_j(t)=0\quad\text{for } |t|\le L_j .
\end{equation}
The regions with nonzero $\sigma_j$ are the so-called PML regions. We choose $L_1>1$ and $L_2>\|p\|_{L^\infty}$ to ensure that the physical region $\Omega_p$ outside the PML encloses the perturbed curve $\Sigma_p$. Put
\begin{equation}\label{eq:DR-SigmaR-def}
  \mathcal D_R=(-L_1-d_1,L_1+d_1)\times(-L_2-d_2,L_2+d_2),
  \qquad
  \Sigma_R=\partial\mathcal D_R .
\end{equation}
We impose the zero Neumann boundary condition on $\Sigma_R$. Thus, $d_j>0$ represents the truncated PML thickness in the $x_j$ direction.  The finite Neumann PML truncation is shown in Figure~\ref{fig:geometry}(b). The truncated impedance boundary is
  $\Gamma_R:=\Sigma\cap\overline{\mathcal D_R}$.
After cutting $\mathcal D_R$ along $\Gamma_R$, the plus-side component is
  $\Omega_R:=\mathcal D_R\cap\mathbb R^2_{\Sigma,+}$,
and the complementary lower component is denoted by $\Omega_R^-$. 

The PML-transformed Helmholtz operator is
\begin{equation}\label{eq:Lsigma}
  \mathcal L_\sigma u:=\nabla\cdot(A_\sigma\nabla u)+k^2J_\sigma u,
\end{equation}
where
  $A_\sigma=\operatorname{diag}(\alpha_2/\alpha_1,\alpha_1/\alpha_2)$ and
  $J_\sigma=\alpha_1\alpha_2$.
The corresponding conormal derivative is
  $\partial_{\nu_\sigma}u:=\nu\cdot A_\sigma\nabla u$,
and we denote
\begin{equation}\label{eq:adjoint-conormal}
  \partial_{\nu_\sigma,y}^{*}G(x,y)
  :=\nu(y)\cdot A_\sigma(y)\nabla_y G(x,y),
\end{equation}
where the notation $\partial_{\nu_\sigma}^{*}$ stresses that the derivative is associated with the source variable $y$. The Neumann PML problem is posed as:
\begin{align}
  \mathcal L_\sigma u_\sigma&=0 &&\text{in }\Omega_R,\label{eq:pmlN-pde}\\
  \partial_{\nu_\sigma}u_\sigma+ M_\beta u_\sigma&=g &&\text{on }\Gamma_R,\label{eq:pmlN-imp}\\
  \partial_{\nu_\sigma}u_\sigma&=0 &&\text{on }\Sigma_R\cap\Omega_R,\label{eq:pmlN-boundary}
\end{align}
where the multiplication operator $M_\beta$ is defined by
 $M_\beta\phi:=(-\ii k\beta)\phi$.

We use the same choices of PML parameters as in \cite{JiangLi2020}.  More precisely,
we take 
\begin{equation}\label{eq:constant-layer-profile}
  \sigma_j(t)=\bar\sigma\,\sigma_j^0(t),\qquad
  \sigma_j^0(t)=d_j^{-1}{\bf 1}_{(L_j,L_j+d_j)\cup(-L_j-d_j,-L_j)}(t),
  \qquad j=1,2.
\end{equation}
such that
\begin{equation}\label{eq:H1-JL}
  \int_{L_j}^{L_j+d_j}\sigma_j(t)\,dt=\bar\sigma,
  \qquad j=1,2,
  \qquad \bar\sigma\ge1 .
\end{equation}
We use the following conventions. 
In the subsequent analysis, \(L_1,L_2,d_1\) are fixed, while \(\bar\sigma\to\infty\) and
\(d_2=d_2(\bar\sigma)\to\infty\) in error estimates. The vertical
thickness $d_2$ is chosen large enough to control the surface-wave factor;
in particular, the simple choice $d_2=\bar{\sigma}$ is sufficient. The generic
constant $C$ and a factor denoted by \(\mathcal P(\bar\sigma)\) may depend on the fixed geometry, on the fixed tangential layer
thickness, on \(k\), \(Z\), and on
\(\beta\), but are independent of the boundary data $g$.  

\subsection{\texorpdfstring{Computable PML-BIE and solution spaces}{Computable PML-BIE and solution spaces}}\label{subsec:compbie}

As mentioned in Introduction, the computable PML-BIE uses the stretched free-space kernel
\begin{equation}\label{eq:Phi-sigma}
  \Phi_\sigma(x,y):=\Phi_k(F_\sigma(x),F_\sigma(y)), \qquad \Phi_k(x,y)=\frac{\ii}{4}H_0^{(1)}(k|x-y|).
\end{equation}
Define the following single-layer and double-layer operators in $\Omega_R$
\begin{align}
  (\mathrm{SL}_\sigma^c q)(x)&:=\int_{\Gamma_R}\Phi_\sigma(x,y)q(y)\,ds_y,\label{eq:S-tilde-intro}\\
  (\mathrm{DL}_\sigma^c\phi)(x)&:=\int_{\Gamma_R}\partial_{\nu_\sigma,y}^{*}\Phi_\sigma(x,y)\phi(y)\,ds_y.
  \label{eq:D-tilde-intro}
\end{align}
The PML-BIE method \cite{LuLuQian2018} uses the following approximate Green identity
\begin{equation}
  u^c_\sigma (x) \approx (\mathrm{SL}_\sigma^c \partial_{\nu_\sigma}u^c_\sigma)(x)  - (\mathrm{DL}_\sigma^c u^c_\sigma)(x) = [\mathrm{SL}_\sigma^c (g - M_\beta u^c_\sigma)](x)  - (\mathrm{DL}_\sigma^c u^c_\sigma)(x), \qquad x\in \Omega_R, \label{eq:appGreen-identity}
\end{equation}
where the fields on the PML boundary $\Sigma_R\cap\Omega_R$ are discarded. Taking the trace as $x$ approaches $\Gamma_R$, we obtain the computable PML-BIE formulation
\begin{equation}\label{eq:computable-bie}
  \left(\frac12I+\mathcal K_\sigma^c+
  \mathcal S_\sigma^c M_\beta\right)\phi_\sigma^c=\mathcal S_\sigma^c g,\qquad{\rm on}\ \Gamma_R,
\end{equation}
where the boundary operators are
\begin{align}
  (\mathcal S_\sigma^c q)(x)&:=\int_{\Gamma_R}\Phi_\sigma(x,y)q(y)\,ds_y,
  \qquad x\in\Gamma_R,\label{eq:V-tilde-intro}\\
  (\mathcal K_\sigma^c\phi)(x)&:=\operatorname{p.v.}\int_{\Gamma_R}\partial_{\nu_\sigma,y}^{*}\Phi_\sigma(x,y)\phi(y)\,ds_y,
  \qquad x\in\Gamma_R,\label{eq:K-tilde-intro}
\end{align}
and $\phi_\sigma^c=u_\sigma^c|_{\Gamma_R}$ is the unknown to be determined. We expect that the resulting field $u_\sigma^c$ in
\eqref{eq:appGreen-identity} converges exponentially to the
outgoing solution of \eqref{eq:phys-pde}--\eqref{eq:phys-src} in the physical
region $\Omega_p$. Consequently, the main task of this paper is to prove unique
solvability of \eqref{eq:computable-bie} and to establish exponential
convergence.

Naturally, the trace spaces used for the boundary integral formulation \eqref{eq:computable-bie} are
\begin{equation}\label{eq:YR-def}
  Y_R:=H^{1/2}(\Gamma_R),
  \qquad
  Y_R':=(H^{1/2}(\Gamma_R))'=H^{-1/2}_{00}(\Gamma_R),
\end{equation}
because the unknown trace should be sought in $Y_R$. Therefore, a Dirichlet PML is unsuitable for the present trace-space formulation, since it provides a trace in the incompatible space $H_{00}^{1/2}(\Gamma_R)$. This is the essential motivation for imposing the Neumann PML above: $u_\sigma$ is naturally sought in $H^1(\Omega_R)$ with the same trace space $Y_R$ on $\Gamma_R$.  The dual space $Y_R'$ is the corresponding conormal-data space.  Since $g$ is compactly supported on $\Gamma_R$, we have $g|_{\Gamma_R}\in Y_R'$. The standard trace mapping properties of the stretched free-space layer operators give
\begin{equation}
  M_\beta: Y_R\to Y_R',\qquad \mathcal S_\sigma^c:Y_R'\to Y_R,
  \qquad
  \mathcal K_\sigma^c:Y_R\to Y_R.
\end{equation}
Thus the computable PML-BIE problem \eqref{eq:computable-bie} can be restated as follows: for any $g\in Y_R'$, find $\phi_\sigma^c\in Y_R$ that satisfies  
\begin{equation}
  \mathcal A_\sigma^c\phi_\sigma^c=\mathcal S_\sigma^c g,
  \qquad
  \mathcal A_\sigma^c:=\frac12I+\mathcal K_\sigma^c+
  \mathcal S_\sigma^c M_\beta:\ Y_R\to Y_R.
\label{eq:computable-bie2}
\end{equation}
Note that only \(L_1,d_1\) and the horizontal layer profile of the PML enter these boundary operators. Later, in Section~\ref{sec:computable}, we seek $\phi_\sigma^c$ in a more restricted subspace of $Y_R$.

To analyze \eqref{eq:computable-bie2}, we first need to establish a related exact PML-BIE formulation for $u_\sigma$ in the Neumann PML problem \eqref{eq:pmlN-pde}--\eqref{eq:pmlN-boundary}.

\section{\texorpdfstring{Convergence theory of the Neumann PML PDE}{Convergence theory of the Neumann PML PDE}}\label{sec:conv-pmlpde}
In this section, we establish the convergence theory of the Neumann PML problem
\eqref{eq:pmlN-pde}--\eqref{eq:pmlN-boundary}. Since the proof relies on
essentially the same techniques as those developed in
\cite{JiangLi2020,ChenZheng2010}, we make only necessary remarks below.
Throughout the rest of this paper, we will by default assume $\bar{\sigma}$
sufficiently large.

Set
\begin{equation}\label{eq:Q0-P-def}
  Q_0:=(-L_1,L_1)\times(-L_2,L_2),\qquad
  Q_R:=\mathcal D_R=(-L_1-d_1,L_1+d_1)\times(-L_2-d_2,L_2+d_2),
\end{equation}
so that the physical region $\Omega_p=Q_0\cap \Omega_R$, and define the full PML layer
\begin{equation}\label{eq:P-def}
  P:=\Omega_R\cap(Q_R\setminus\overline{Q_0}).
\end{equation}
Thus \(P\) contains the horizontal, vertical, and corner PML subregions.  The boundary pieces used in the layer estimates are
\begin{equation}\label{eq:pml-layer-boundary-pieces}
  \Gamma_{\rm in}:=\partial Q_0\cap\mathbb R^2_{\Sigma,+},\qquad
  \Gamma_{\rm out}:=\partial Q_R\cap\mathbb R^2_{\Sigma,+},\qquad
  \Sigma_{\rm pml}:=\Sigma\cap(Q_R\setminus\overline{Q_0}).
\end{equation}

The next three lemmas study some important properties in the PML layer $P$.
\begin{lemma}\label{lem:neumann-layer-correction}
For any \(r\in H^{-1/2}(\Gamma_{\rm out})\), the problem
\begin{align}
  \mathcal L_\sigma \eta&=0 &&\text{in }P,\label{eq:layer-corr-pde}\\
  \eta&=0 &&\text{on }\Gamma_{\rm in},\label{eq:layer-corr-inner}\\
  \partial_{\nu_\sigma}\eta&=r &&\text{on }\Gamma_{\rm out},\label{eq:layer-corr-outer}\\
  \partial_{\nu_\sigma}\eta-\ii k\beta\eta&=0 &&\text{on }\Sigma_{\rm pml}\label{eq:layer-corr-imp}
\end{align}
has a unique solution in \(V_P:=\{v\in H^1(P): v=0\text{ on }\Gamma_{\rm in}\}\), and
\begin{equation}\label{eq:layer-corr-est}
  \|\eta\|_{H^1(P)}+
  \|\partial_{\nu_\sigma}\eta|_{\Gamma_{\rm in}}\|_{H^{-1/2}(\Gamma_{\rm in})}
  \le \mathcal P(\bar\sigma)\|r\|_{H^{-1/2}(\Gamma_{\rm out})}.
\end{equation}
The same estimate holds if the last condition is replaced by \(\eta=0\) on \(\Sigma_{\rm pml}\).
\end{lemma}

\begin{proof}
The well-posedness proof follows essentially the same argument as Lemma~5.1 in \cite{JiangLi2020}. The only difference is that Friedrichs' inequality is applied directly in the whole PML layer $P$, rather than only in the corner PML regions, because the Neumann boundary condition \eqref{eq:layer-corr-outer} is now imposed. We thus obtain
\begin{equation}
  \|\eta\|_{H^1(P)}
  \le \mathcal P(\bar\sigma)
       \|r\|_{H^{-1/2}(\Gamma_{\rm out})}.
\end{equation}
The conormal trace on \(\Gamma_{\rm in}\) is estimated by the variational definition: for \(\psi\in H^{1/2}(\Gamma_{\rm in})\), take an extension \(V\in H^1(P)\) with \(V|_{\Gamma_{\rm in}}=\psi\) and \(V=0\) on \(\Gamma_{\rm out}\).  Testing the problem with this extension leads to 
\begin{equation}
  |\langle\partial_{\nu_\sigma}\eta,\psi\rangle_{\Gamma_{\rm in}}|
  \le C(1+\|\beta\|_{L^\infty(\Sigma_{\rm pml})})
      \|\eta\|_{H^1(P)}\|V\|_{H^1(P)}
  \le C\|\eta\|_{H^1(P)}\|\psi\|_{H^{1/2}(\Gamma_{\rm in})}.
\end{equation}
This proves the second estimate in \eqref{eq:layer-corr-est}.  The remaining details are omitted.
\end{proof}

\begin{lemma}\label{lem:stretched-green-estimates}
Let \(E(f)\) be the complex-stretched extension in the PML layer $P$ associated with boundary data \(f\in H^{1/2}(\Gamma_{\rm in})\). There exists \(c>0\), independent of $\bar{\sigma}$, such that
\begin{equation}\label{eq:JL-outer-conormal-summary}
  \|\partial_{\nu_\sigma}E(f)|_{\Gamma_{\rm out}}\|_{H^{-1/2}(\Gamma_{\rm out})}
  \le \mathcal P(\bar\sigma)
     \left(e^{-ck\bar\sigma}+e^{-Z(L_2+d_2)}\right)
     \|f\|_{H^{1/2}(\Gamma_{\rm in})} .
\end{equation}
Moreover, if \(\chi_{\rm tail}\) is supported in $\overline{\Sigma_{\rm pml}}$
with support well-separated from $\overline{\Omega_p}$, then 
\begin{equation}\label{eq:JL-tail-summary}
  \|\chi_{\rm tail}E(f)\|_{Y_R}
  +\|\chi_{\rm tail}\partial_{\nu_\sigma}E(f)\|_{Y_R'}
  \le \mathcal P(\bar\sigma)e^{-c\bar\sigma}
     \|f\|_{H^{1/2}(\Gamma_{\rm in})} .
\end{equation}
\end{lemma}
\begin{proof}
These follow from the estimates of the impedance half-space Green function in the PML region $P$, as developed in \cite[Sec.~4]{JiangLi2020}.  
\end{proof}

\begin{lemma}\label{lem:neumann-dtn-perturbation}
Let \(T\) be the exact DtN map on \(\Gamma_{\rm in}\), defined by the impedance-line outgoing problem, and let \(T_N^\sigma\) be the finite-layer DtN map obtained by solving in the PML layer \(P\) with the impedance condition on \(\Sigma_{\rm pml}\) and the homogeneous conormal Neumann condition on \(\Gamma_{\rm out}\).  Then, 
\begin{equation}\label{eq:neumann-dtn-perturbation}
  \|T_N^\sigma-T\|_{H^{1/2}(\Gamma_{\rm in})\to H^{-1/2}(\Gamma_{\rm in})}
  \le \mathcal P(\bar\sigma)
      \left(e^{-c k\bar\sigma}+e^{-Z(L_2+d_2)}\right).
\end{equation}
\end{lemma}

\begin{proof}
The proof is similar to the proof of Lemma~5.2 in \cite{JiangLi2020}.
The only difference is that one must now estimate the conormal residual $\partial_{\nu_\sigma}E(f)$ on $\Gamma_{\rm out}$, which is done in Lemma~\ref{lem:stretched-green-estimates}. Applying Lemma~\ref{lem:neumann-layer-correction} with $r=-\partial_{\nu_\sigma}E(f)$ completes the proof.

\end{proof}

Combining Lemmas~\ref{lem:neumann-layer-correction}, \ref{lem:stretched-green-estimates}, and \ref{lem:neumann-dtn-perturbation}, we obtain the following PML-PDE convergence theorem.

\begin{theorem}\label{thm:pml-pde}
The Neumann PML problem \eqref{eq:pmlN-pde}--\eqref{eq:pmlN-boundary} is uniquely solvable.  If \(u\) is the outgoing solution of \eqref{eq:phys-pde}--\eqref{eq:phys-src} and \(u_\sigma\) is the Neumann PML solution, then, for any compact set \(K\Subset\Omega_p\),
\begin{equation}\label{eq:pde-conv}
  \|u-u_\sigma\|_{H^1(K)}
  \le \mathcal P(\bar\sigma)\left(e^{-c_0k\bar\sigma}+e^{-Z(L_2+d_2)}\right)\|g\|_{Y_R'} .
\end{equation}
Let \(\chi_{\rm tail}\) be defined in Lemma~\ref{lem:stretched-green-estimates}. Then, there exists \(\alpha>0\), independent of $\bar\sigma$, such that
\begin{equation}\label{eq:pml-terminal-tail-trace-decay}
  \|\chi_{\rm tail}u_\sigma|_{\Gamma_R}\|_{Y_R}+
  \|\chi_{\rm tail}\partial_{\nu_\sigma}u_\sigma|_{\Gamma_R}\|_{Y_R'}
  \le C e^{-\alpha\bar\sigma}\|g\|_{Y_R'} .
\end{equation}
Moreover, if \(w\in H^1(\Omega_R)\) solves the forced Neumann PML problem
\begin{align}
  \mathcal L_\sigma w&=F &&\text{in }\Omega_R,\label{eq:forced-robin-pml-pde}\\
  \partial_{\nu_\sigma}w+M_\beta w&=r &&\text{on }\Gamma_R,\label{eq:forced-robin-pml-bc}\\
  \partial_{\nu_\sigma}w&=0 &&\text{on }\Sigma_R,\label{eq:forced-robin-pml-boundary}
\end{align}
with \(F\in H^{-1}_{0}(\Omega_R)\) and \(r\in Y_R'\), then
\begin{equation}\label{eq:forced-robin-pml-est}
  \|w\|_{H^1(\Omega_R)}+
  \|w|_{\Gamma_R}\|_{Y_R}+
  \|\partial_{\nu_\sigma}w|_{\Gamma_R}\|_{Y_R'}
  \le \mathcal P(\bar\sigma)
  \big(\|F\|_{H^{-1}_{0}(\Omega_R)}+\|r\|_{Y_R'}\big).
\end{equation}
\end{theorem}

\begin{proof}
The proof of \eqref{eq:pde-conv} and \eqref{eq:forced-robin-pml-est} for $F$ supported in $\overline{\Omega_p}$
is routine by following \cite{JiangLi2020}. If the source \(F\) is supported in the
PML layer, one first solves the corresponding layer problem with zero data on
the PML entrance, extends this layer solution by zero into the physical region
$\Omega_p$, and subtracts it.  The remaining problem has a source and boundary
data supported in $\Omega_p$; the already established Neumann PML stability then
gives \eqref{eq:forced-robin-pml-est}.  The case of a general \(F\) follows by
splitting \(F=F_{\rm phys}+F_{\rm pml}\).  We focus only on estimating the
PML tails in \eqref{eq:pml-terminal-tail-trace-decay}.  Let
\(f=u_\sigma|_{\Gamma_{\rm in}}\), let \(E(f)\) be the stretched
extension, and write \(u_\sigma=E(f)+\eta_f\) in the PML layer.
Lemma~\ref{lem:stretched-green-estimates} gives the tail decay of \(E(f)\) and
the exponentially small outer conormal residual.
Lemma~\ref{lem:neumann-layer-correction} applied to the correction \(\eta_f\)
gives the same tail bound for \(\eta_f\), up to a polynomial factor.  After
decreasing the exponent to absorb this polynomial factor,
\eqref{eq:pml-terminal-tail-trace-decay} follows.
\end{proof}

We next prove well-posedness for an auxiliary PML problem with a Dirichlet boundary condition on $\Gamma_R$ that will be used frequently later. 
\begin{lemma}
\label{lem:aux-mixed-unique}
Let \(D=\Omega_R\) or \(D=\Omega_R^-\), and put \(\Sigma_R^D:=\partial D\cap\Sigma_R\).  The mixed problem
\begin{align}
  \mathcal L_\sigma w_\phi&=F &&\text{in }D,\label{eq:aux-dir-pde}\\
  w_\phi&=\phi &&\text{on }\Gamma_R,\label{eq:aux-dir-data}\\
  \partial_{\nu_\sigma}w_\phi&=0 &&\text{on }\Sigma_R^D\label{eq:aux-dir-boundary}
\end{align}
has a unique solution for any \(\phi\in H^{1/2}(\Gamma_R)\), and
\begin{equation}\label{eq:aux-mixed-bound}
  \|w_\phi\|_{H^1(D)}+
  \|\partial_{\nu_\sigma}w_\phi|_{\Gamma_R}\|_{H^{-1/2}(\Gamma_R)}
  \le \mathcal P(\bar\sigma)\big(\|\phi\|_{H^{1/2}(\Gamma_R)}+\|F\|_{H^{-1}_{0}(D)}\big).
\end{equation}
\end{lemma}

\begin{proof}

The proof is quite similar to that for the impedance condition 
\eqref{eq:pmlN-imp} on \(\Gamma_R\) in Theorem~\ref{thm:pml-pde}. We only emphasize that the corresponding
untruncated problem is the half-space Dirichlet problem for scattering in a locally perturbed half-space; its uniqueness follows from the result of Chandler-Wilde and Monk
\cite{ChandlerWildeMonk2005}. We omit the details.

\end{proof}

\section{\texorpdfstring{The exact Neumann PML-Green BIE}{The exact Neumann PML-Green BIE}}\label{sec:exact-green-pde}
In this section we derive the exact PML-Green BIE and establish its equivalence with the PML PDE problem \eqref{eq:pmlN-pde}--\eqref{eq:pmlN-boundary}.

\subsection{\texorpdfstring{The exact Neumann PML-Green function}{The exact Neumann PML-Green function}}\label{sec:exact-green}

The exact Neumann PML Green function \(G_{\sigma,N}^0\) is the Green function of the stretched background operator in \(\mathcal D_R\) with homogeneous conormal Neumann condition on the artificial boundary:
\begin{align}
  \mathcal L_{\sigma,x}G_{\sigma,N}^0(x,y)&=-\delta_y(x)&&\text{in }\mathcal D_R,\label{eq:GN-def}\\
  \partial_{\nu_\sigma,x}G_{\sigma,N}^0(x,y)&=0&&\text{on }\Sigma_R.\label{eq:GN-boundary}
\end{align}
 The role of $G_{\sigma,N}^0$ is to insert a genuinely well-posed exact PML-BIE between the original problem \eqref{eq:phys-pde}--\eqref{eq:phys-src} and the computable PML-BIE \eqref{eq:computable-bie2}.  In the rectangular homogeneous background used here, $G_{\sigma,N}^0$ can be constructed via an even reflected-image series.  With \(X=F_\sigma(x)\), \(Y=F_\sigma(y)\), set
\begin{equation}\label{eq:stretched-half-lengths}
  A_j:=F_{\sigma,j}(L_j+d_j)
  =\int_0^{L_j+d_j}\alpha_j(t)\,dt,
  \qquad T_j:=2A_j .
\end{equation}
Since \(\sigma_j\) is even, \(F_{\sigma,j}(-L_j-d_j)=-A_j\).  Thus the PML box is transformed into the rectangle \((-A_1,A_1)\times(-A_2,A_2)\) in stretched coordinates, and the period associated with the even extension in the \(j\)-th coordinate is \(2T_j=4A_j\).  The reflected points are
\begin{equation}\label{eq:reflected-points-corrected}
  (Y_{\ell,\delta})_j=
  \begin{cases}
  Y_j+2\ell_jT_j,&\delta_j=0,\\
  2A_j-Y_j+2\ell_jT_j,&\delta_j=1,
  \end{cases}
  \qquad \ell\in\mathbb Z^2,
  \quad \delta\in\{0,1\}^2.
\end{equation}
A direct verification shows that
\begin{equation}\label{eq:neumann-image-series}
  G_{\sigma,N}^0(x,y)=
  \sum_{\ell\in\mathbb Z^2}\sum_{\delta\in\{0,1\}^2}
  \Phi_k(X,Y_{\ell,\delta}),
\end{equation}
solves the problem~\eqref{eq:GN-def}--\eqref{eq:GN-boundary}. 
Write 
\begin{equation}\label{eq:GN-local-splitting}
  G_{\sigma,N}^0(x,y)=\Phi_\sigma(x,y)+R_{\sigma,N}(x,y),
\end{equation}
where $\Phi_\sigma(x,y)$ is the singular incident field for the term \((\ell,\delta)=(0,0)\) and \(R_{\sigma,N}\) is the regular scattering field for the remaining terms. 
This construction is the Neumann analogue of the UPML Green functions used in
\cite{LuLaiWu2024}. It remains to justify the convergence of the infinite series
in \eqref{eq:neumann-image-series}.

\begin{proposition}\label{prop:reflected-kernel-estimate}
Let a compact set \(K\Subset\Omega_R\) and let \(\overline{\Gamma_R^{\rm
int}}\Subset\Gamma_R\).  For
any \(m\ge0\), there exist constants \(C_m,\alpha_m>0\) such that
\begin{equation}
  \|R_{\sigma,N}\|_{C^m(K\times\Gamma_R)}+
  \|\partial_{\nu_\sigma,y}^*R_{\sigma,N}\|_{C^m(K\times\Gamma_R)}
  \le C_m e^{-\alpha_m\bar\sigma}.
  \label{eq:obs-kernel-small}
\end{equation}
Moreover,
\begin{equation}\label{eq:boundary-source-interior-kernel-small}
  \|R_{\sigma,N}\|_{C^m(\Gamma_R\times\Gamma_R^{\rm int})}+
  \|\partial_{\nu_\sigma,y}^*R_{\sigma,N}\|_{C^m(\Gamma_R\times\Gamma_R^{\rm int})}
  \le C_m e^{-\alpha_m\bar\sigma}.
\end{equation}
\end{proposition}

\begin{proof}
We use the large-argument part of the Hankel estimates from
\cite{ChenLiu2005,ChenZheng2010}.  More precisely, for any fixed
\(\rho_0>0\),
\begin{equation}
  |\partial_z^\mu H_0^{(1)}(kz)|
  \le C_{\mu,\rho_0} (1+|kz|)^{-1/2}|z|^{-\mu}e^{-\Im(kz)},
  \qquad |z|\ge \rho_0,
  \quad \mu=0,1,\ldots,m+1.\label{eq:Greendecay}
\end{equation}
The restriction \(|z|\ge\rho_0\) is available for
all reflected terms.  Indeed, either the observation set is a fixed compact set
\(K\Subset\Omega_R\), or the source variable is restricted to
\(\Gamma_R^{\rm int}\Subset\Gamma_R\); in both cases each reflected image point
stays a positive Euclidean distance from the target set.  In the reflected series
\eqref{eq:neumann-image-series}, each nontrivial reflected path then
passes through a positive amount of PML layer before returning to the
observation region.  Hence the reflected series and its derivatives admit a
summable majorant of size \(e^{-c\bar\sigma}\), after absorbing algebraic
PML factors into the constant and decreasing the exponent if necessary.
\end{proof}

We remark that if both variables approach $\partial G_R$, a reflected source can
approach the physical region $\Omega_p$. Therefore, the above argument does not
yield an exponentially small operator-norm estimate on all of
\(\Gamma_R\times\Gamma_R\).

\subsection{\texorpdfstring{The exact PML-Green BIE formulation}{The exact PML-Green BIE formulation}}\label{sec:exact-bie}

For \(q\in Y_R'\) and \(\phi\in Y_R\), by analogy with the computable layer
potentials $\mathrm{SL}_{\sigma}^c$ and $\mathrm{DL}_{\sigma}^c$, we define the
exact PML-Green layer potentials $\mathrm{SL}_{\sigma,N}^0$ and $\mathrm{DL}_{\sigma,N}^0$
by replacing $\Phi_\sigma$ by $G_{\sigma,N}^0$.  Their boundary operators
$\mathcal S_{\sigma,N}^0$ and $\mathcal K_{\sigma,N}^0$ are defined analogously
by \eqref{eq:V-tilde-intro} and \eqref{eq:K-tilde-intro}, with $G_{\sigma,N}^0$
as the new kernel.  They are well defined using the standard Sobolev spaces on
open arcs, and are bounded
as follows:
\begin{equation}
  \mathcal S_{\sigma,N}^0:Y_R'\to Y_R,
  \qquad
  \mathcal K_{\sigma,N}^0:Y_R\to Y_R;
\end{equation}
see \cite[Chs.~6--8]{McLean2000} for details and for the corresponding jump relations.

In the following, the plus trace $\gamma_+$ is taken from \(\Omega_R\), and the minus trace $\gamma_-$ from the lower component \(\Omega_R^-\).  The conormal direction is still the same geometric normal \(\nu\), as shown in Figure~\ref{fig:geometry}.
Let \(u_\sigma\) solve \eqref{eq:pmlN-pde}--\eqref{eq:pmlN-boundary}.  Set
\begin{equation}
  \phi_\sigma=u_\sigma|_{\Gamma_R},
  \qquad
  q_\sigma=\partial_{\nu_\sigma}u_\sigma|_{\Gamma_R}.
\end{equation}
For \(x\in\Omega_R\), Green's representation theorem gives
\begin{align}
  u_\sigma(x)
  &=(\mathrm{SL}_{\sigma,N}^0 q_\sigma)(x)  - (\mathrm{DL}_{\sigma,N}^0 \phi_\sigma)(x),
    \label{eq:green-rep-pde}
\end{align}
where we have used the fact that \(\partial_{\nu_\sigma}u_\sigma=0\) and \(\partial_{\nu_\sigma,y}^{*}G_{\sigma,N}^0=0\) on \(\Sigma_R\). Letting $x$ approach $\Gamma_R$ yields
\begin{equation}\label{eq:bie-q}
  \left(\frac12I+\mathcal K_{\sigma,N}^0\right)\phi_\sigma
  =\mathcal S_{\sigma,N}^0q_\sigma.
\end{equation}
Since the impedance boundary condition is
\begin{equation}
  q_\sigma+M_\beta\phi_\sigma=g,
\end{equation}
we obtain the exact PML-Green BIE
\begin{equation}
  \mathcal A_{\sigma,N}^0\phi_\sigma=\mathcal S_{\sigma,N}^0g,
  \qquad
  \mathcal A_{\sigma,N}^0=\frac12I+\mathcal K_{\sigma,N}^0+\mathcal S_{\sigma,N}^0M_\beta:Y_R\to Y_R .
  \label{eq:exact-bie}
\end{equation}

\subsection{\texorpdfstring{Equivalence of the exact PML-Green BIE and the PML-PDE}{Equivalence of the exact PML-Green BIE and the PML-PDE}}

For \(\phi\in Y_R\), let \(w_\phi\) be the auxiliary mixed solution from Lemma~\ref{lem:aux-mixed-unique}, and define the Neumann PML Dirichlet-to-Neumann (DtN) map
\begin{equation}\label{eq:LambdaN-def}
  \Lambda_{\sigma,N}^0\phi:=\partial_{\nu_\sigma}w_\phi|_{\Gamma_R}\in Y_R'.
\end{equation}
Applying \eqref{eq:green-rep-pde} to \(w_\phi\), with \(q=\Lambda_{\sigma,N}^0\phi\), gives
\begin{equation}\label{eq:wphi-rep-N}
  w_\phi=\mathrm{SL}_{\sigma,N}^0\Lambda_{\sigma,N}^0\phi
  -\mathrm{DL}_{\sigma,N}^0\phi .
\end{equation}
Taking the plus trace and using \(\gamma_+w_\phi=\phi\) gives
\begin{equation}\label{eq:factor1-N}
  \left(\frac12I+\mathcal K_{\sigma,N}^0\right)\phi
  =\mathcal S_{\sigma,N}^0\Lambda_{\sigma,N}^0\phi .
\end{equation}
Consequently,
\begin{equation}\label{eq:factor-A-N}
  \mathcal A_{\sigma,N}^0=\mathcal S_{\sigma,N}^0(\Lambda_{\sigma,N}^0+M_\beta).
\end{equation}
We first show the injectivity of ${\cal S}_{\sigma,N}^0$. 
\begin{lemma}\label{lem:S-exact-injective}
The operator \(\mathcal S_{\sigma,N}^0:Y_R'\to Y_R\) is injective.
\end{lemma}

\begin{proof}
Let \(q\in Y_R'\) satisfy \(\mathcal S_{\sigma,N}^0q=0\), and set
  $v=\mathrm{SL}_{\sigma,N}^0q$.
Then \(v\) satisfies \(\mathcal L_\sigma v=0\) in \(\mathcal D_R\setminus\overline{\Gamma_R}\).  Since the exact PML-Green kernel has homogeneous conormal Neumann trace on \(\Sigma_R\),
\begin{equation}
  \partial_{\nu_\sigma}v=0\qquad\text{on }\Sigma_R.
\end{equation}
The single-layer potential is continuous across \(\Gamma_R\), and its common trace is \(\mathcal S_{\sigma,N}^0q=0\).  Therefore the restrictions \(v^+\) and \(v^-\) to \(\Omega_R\) and \(\Omega_R^-\) are the zero-trace auxiliary mixed solutions in Lemma~\ref{lem:aux-mixed-unique}.  Hence
\begin{equation}
  v^+=v^-=0.
\end{equation}
Their conormal traces vanish on \(\Gamma_R\).  The jump relation gives
\begin{equation}
  q=\partial_{\nu_\sigma,+}v-
  \partial_{\nu_\sigma,-}v=0.
\end{equation}
\end{proof}

We now prove the equivalence between the PML problem \eqref{eq:pmlN-pde}--\eqref{eq:pmlN-boundary} and the exact PML-Green BIE \eqref{eq:exact-bie}. 
\begin{theorem}\label{thm:equiv}
The function \(\phi_\sigma\in Y_R\) solves the exact PML-Green BIE \eqref{eq:exact-bie} if and only if \(\phi_\sigma=u_\sigma|_{\Gamma_R}\), where \(u_\sigma\) is the solution of the Neumann PML problem \eqref{eq:pmlN-pde}--\eqref{eq:pmlN-boundary}. Therefore,
\begin{equation}
  u_\sigma=\mathrm{SL}_{\sigma,N}^0(g-M_\beta\phi_\sigma)-\mathrm{DL}_{\sigma,N}^0\phi_\sigma .
  \label{eq:exact-reconstruction}
\end{equation}
\end{theorem}

\begin{proof}
The implication from the PDE to the BIE was derived in \eqref{eq:green-rep-pde}--\eqref{eq:exact-bie}.  Conversely, suppose that \(\phi_\sigma\in Y_R\) solves \eqref{eq:exact-bie}.  Using the factorization \eqref{eq:factor-A-N},
\begin{equation}
  \mathcal S_{\sigma,N}^0(\Lambda_{\sigma,N}^0+M_\beta)\phi_\sigma
  =\mathcal S_{\sigma,N}^0g .
\end{equation}
Hence
\begin{equation}
  \mathcal S_{\sigma,N}^0\left((\Lambda_{\sigma,N}^0+M_\beta)\phi_\sigma-g\right)=0 .
\end{equation}
By Lemma~\ref{lem:S-exact-injective},
\begin{equation}\label{eq:LambdaN-Robin}
  (\Lambda_{\sigma,N}^0+M_\beta)\phi_\sigma=g
  \qquad\text{in }Y_R'.
\end{equation}
Let \(w_{\phi_\sigma}\) be the auxiliary mixed solution with trace
\(\phi=\phi_\sigma\).  By the definition of \(\Lambda_{\sigma,N}^0\),
\eqref{eq:LambdaN-Robin} is exactly
\begin{equation}
  \partial_{\nu_\sigma}w_{\phi_\sigma}-\ii k\beta w_{\phi_\sigma}=g
  \qquad\text{on }\Gamma_R.
\end{equation}
Together with \(\mathcal L_\sigma w_{\phi_\sigma}=0\) in \(\Omega_R\) and \(\partial_{\nu_\sigma}w_{\phi_\sigma}=0\) on \(\Sigma_R\), this is the Neumann PML problem.  By uniqueness,
\begin{equation}
  w_{\phi_\sigma}=u_\sigma.
\end{equation}
Finally, \eqref{eq:LambdaN-Robin} gives \(g-M_\beta\phi_\sigma=\Lambda_{\sigma,N}^0\phi_\sigma\).  Substituting this into \eqref{eq:wphi-rep-N} yields the reconstruction formula \eqref{eq:exact-reconstruction}.  The impedance conormal trace is recovered from the auxiliary DtN factorization.
\end{proof}

\subsection{\texorpdfstring{Well-posedness and convergence theory}{Well-posedness and convergence theory}}\label{sec:exact-convergence}

The PML-PDE convergence theorem was established in
Section~\ref{sec:conv-pmlpde}.  We now transfer that result to the exact
PML-Green BIE through the equivalence theorem.

\begin{theorem}\label{thm:exact-conv}
The exact PML-Green BIE \eqref{eq:exact-bie} has a unique solution \(\phi_\sigma\in Y_R\).  Its reconstructed field via \eqref{eq:exact-reconstruction} equals the Neumann PML solution $u_\sigma$.  Consequently, for any compact set \(K\Subset\Omega_p\),
\begin{equation}
  \|u-u_\sigma\|_{H^1(K)}
  \le \mathcal P(\bar\sigma)\left(e^{-c_0k\bar\sigma}+e^{-Z(L_2+d_2)}\right)\|g\|_{Y_R'}.
\end{equation}
\end{theorem}

\begin{proof}
By Theorem~\ref{thm:equiv}, solving the exact PML-Green BIE \eqref{eq:exact-bie} is equivalent to solving the Neumann PML problem \eqref{eq:pmlN-pde}--\eqref{eq:pmlN-boundary}.  The latter is uniquely solvable by Theorem~\ref{thm:pml-pde}.  The reconstruction identity is \eqref{eq:exact-reconstruction}, and the error estimate is exactly \eqref{eq:pde-conv} applied to this reconstructed field.
\end{proof}

To conclude this section, we present two corollaries showing the polynomial stability of the exact PML-Green BIE~\eqref{eq:exact-bie}.
\begin{corollary}\label{cor:exact-bie-stability}
The operator
\(\mathcal A_{\sigma,N}^0:Y_R\to Y_R\) satisfies
\begin{equation}\label{eq:exact-bie-stability}
  \|\phi\|_{Y_R}\le \mathcal P(\bar\sigma)
  \|\mathcal A_{\sigma,N}^0\phi\|_{Y_R},
  \qquad \phi\in Y_R .
\end{equation}
Moreover, 
\begin{equation}\label{eq:exact-bie-stability-rho-unweighted}
  \|\phi\|_{Y_R}+\|\Lambda_{\sigma,N}^0\phi\|_{Y_R'}
  \le \mathcal P(\bar\sigma)
  \|\mathcal A_{\sigma,N}^0\phi\|_{Y_R},
  \qquad \phi\in Y_R .
\end{equation}
\end{corollary}

\begin{proof}
Put $r:=(\Lambda_{\sigma,N}^0+M_\beta)\phi\in Y_R'$.
By the definition of the DtN map \eqref{eq:LambdaN-def}, the auxiliary solution
\(w_\phi\) from Lemma~\ref{lem:aux-mixed-unique} has Dirichlet trace \(\phi\) on
\(\Gamma_R\) and conormal trace \(\Lambda_{\sigma,N}^0\phi\).  Hence $r$ is
precisely the Robin datum of \(w_\phi\) on \(\Gamma_R\).  Applying
\eqref{eq:forced-robin-pml-est} in Theorem~\ref{thm:pml-pde} with $F=0$ and
datum \(r\) gives
\begin{equation}\label{eq:lambda-robin-stability}
  \|\phi\|_{Y_R}
  +\|\Lambda_{\sigma,N}^0\phi\|_{Y_R'}
  \le \mathcal P(\bar\sigma)\|r\|_{Y_R'} .
\end{equation}
It remains to bound \(r\) through its exact single-layer trace.  Define
\begin{equation}\label{eq:exact-stability-single-layer-field}
  v=\mathrm{SL}_{\sigma,N}^0 r,
  \qquad f=\mathcal S_{\sigma,N}^0r .
\end{equation}
By the trace continuity, both restrictions \(v^+\) and \(v^-\) have the same Dirichlet trace \(f\) on \(\Gamma_R\). Applying the auxiliary estimate \eqref{eq:aux-mixed-bound} of Lemma~\ref{lem:aux-mixed-unique} to the two sides gives
\begin{equation}\label{eq:S-lower-pde}
  \|\partial_{\nu_\sigma,+}v\|_{Y_R'}+
  \|\partial_{\nu_\sigma,-}v\|_{Y_R'}
  \le \mathcal P(\bar\sigma)\|f\|_{Y_R} .
\end{equation}
This, together with \eqref{eq:exact-stability-single-layer-field}, implies
\begin{equation}\label{eq:exact-S-lower}
  \|r\|_{Y_R'} = \|\partial_{\nu_\sigma,+}v-
  \partial_{\nu_\sigma,-}v\|_{Y_R'}
  \le \mathcal P(\bar\sigma)
  \|\mathcal S_{\sigma,N}^0r\|_{Y_R} .
\end{equation}
Finally, the factorization \eqref{eq:factor-A-N} and the definition of $r$ give
\begin{equation}\label{eq:exact-stability-A-as-Sr}
  \mathcal A_{\sigma,N}^0\phi=\mathcal S_{\sigma,N}^0r .
\end{equation}
Combining \eqref{eq:lambda-robin-stability}, \eqref{eq:exact-S-lower}, and \eqref{eq:exact-stability-A-as-Sr} proves both \eqref{eq:exact-bie-stability} and \eqref{eq:exact-bie-stability-rho-unweighted}.
\end{proof}

\begin{corollary}\label{cor:exact-physical-family-bound}
Let \(\phi_\sigma^0\) solve the exact PML-Green BIE \eqref{eq:exact-bie}.  Then,
\begin{equation}\label{eq:exact-physical-unweighted}
  \|\phi_\sigma^0\|_{Y_R}
  \le \mathcal P(\bar\sigma)\|g\|_{Y_R'} .
\end{equation}
Moreover, for $\chi_{\rm tail}$  defined in Lemma~\ref{lem:stretched-green-estimates}, 
\begin{equation}\label{eq:exact-physical-family-tail-small}
  \|\chi_{\rm tail}\phi_\sigma^0\|_{Y_R}
  \le \mathcal P(\bar\sigma)e^{-\alpha\bar\sigma}\|g\|_{Y_R'} .
\end{equation}
\end{corollary}

\begin{proof}
The bound follows from the exact PML-Green BIE stability \eqref{eq:exact-bie-stability}
and the boundedness of \(\mathcal S_{\sigma,N}^0:Y_R'\to Y_R\):
\begin{equation}\label{eq:cor-unweighted}
  \|\phi_\sigma^0\|_{Y_R}
  \le \mathcal P(\bar\sigma)\|\mathcal S_{\sigma,N}^0g\|_{Y_R}
  \le C\mathcal P(\bar\sigma)\|g\|_{Y_R'} .
\end{equation}
By the equivalence between the exact PML-Green BIE and the Neumann PML problem,
\(\phi_\sigma^0=u_\sigma|_{\Gamma_R}\).  Applying
\eqref{eq:pml-terminal-tail-trace-decay} directly gives
\eqref{eq:exact-physical-family-tail-small}.
\end{proof}

\section{\texorpdfstring{The computable PML-BIE}{The computable PML-BIE}}\label{sec:computable}

This section proves the convergence of the computable stretched-kernel PML-BIE.  The exact PML-Green BIE \eqref{eq:exact-bie} in Section~\ref{sec:exact-bie} is
\[
  \mathcal A_{\sigma,N}^0\phi_\sigma^0=\mathcal S_{\sigma,N}^0g,
\]
whereas the computable equation \eqref{eq:computable-bie2} in Section~\ref{subsec:compbie} uses the stretched free-space kernel,
\[
  \mathcal A_\sigma^c\phi_\sigma^c=\mathcal S_\sigma^c g .
\]
Therefore,
\[
  \mathcal A_{\sigma,N}^0\phi_\sigma^c =  \mathcal S_\sigma^c g + (\mathcal B_\sigma^{\rm im} + \mathcal E_\sigma)\phi_\sigma^c,
\]
so that 
\[
  \phi_\sigma^c = (\mathcal A_{\sigma,N}^0)^{-1}\mathcal S_\sigma^c g + (\mathcal A_{\sigma,N}^0)^{-1}(\mathcal B_\sigma^{\rm im} + \mathcal E_\sigma)\phi_\sigma^c,
\]
where a localized nearest-image part \(\mathcal B_\sigma^{\rm im}\) and a
separated remainder \(\mathcal E_\sigma\) are defined later in
Subsection~\ref{subsec:nearest-image-computable}. The purpose of this section is to justify the
unique solvability of \eqref{eq:computable-bie2} and to prove its exponential
convergence to the exact PML-Green BIE solution $\phi_\sigma^0$,  by showing that 
\begin{equation}
  \label{eq:AinvBE}
  \|(\mathcal A_{\sigma,N}^0)^{-1} \mathcal B_\sigma^{\rm im}\|\to0,\qquad \bar{\sigma}\to\infty,
\end{equation}
in a weighted space norm. The remainder $\mathcal E_\sigma$ is in fact exponentially small.

\subsection{\texorpdfstring{Weighted spaces and exact inverse identity}{Weighted spaces and exact inverse identity}}\label{sec:weighted-exact-stability}

Choose
\(\chi_{\rm im}\in C^\infty(\Gamma_R)\), \(0\le \chi_{\rm im}\le1\), whose support is contained in the two flat PML segments $\overline{\Sigma_{\rm pml}}$ adjacent to the vertical Neumann boundaries at $x_1=\pm(L_1+d_1)$, and set \(\Gamma_{\rm im}:=\operatorname{supp}\chi_{\rm im}\).  On each component of \(\Gamma_{\rm im}\) we use local coordinates
  $(s,n)$ for $s>0$,
where \(s\) is the distance to the adjacent vertical boundary and \(n\) is the normal coordinate to \(\Gamma_R\).  The boundary is \(s=0\), and the PML part of \(\Gamma_R\) is \(n=0\).  By the piecewise-constant PML profile \eqref{eq:constant-layer-profile}, the tangential stretching coefficient \(\alpha_1 = 1 + \ii \bar{\sigma}/d_1\) is constant on PML segments. 
Let \(c_0=\alpha>0\) be the decay rate introduced in Corollary~\ref{cor:exact-physical-family-bound} with $\chi_{\rm tail}=\chi_{\rm im}$.
For \(a\in(0,c_0)\), put
\begin{equation}\label{eq:W-def}
  W_\sigma:=e^{a\bar\sigma}.
\end{equation}
 For fixed \(\bar\sigma\), we define the weighted norm
\begin{equation}\label{eq:X-fixed-norm}
  \|h_\sigma\|_{Y_\sigma^{a}}
  :=\|h_\sigma\|_{Y_R}+W_\sigma\|\chi_{\rm im}h_\sigma\|_{Y_R} .
\end{equation}
 The weighted family norm associated with \(Y_\sigma^{a}\) is
\begin{equation}\label{eq:X-family-norm}
  \|\mathbf h\|_{\mathfrak Y^{a}}
  :=\sup_{\bar\sigma\ge\bar\sigma_0}
  \left(\|h_\sigma\|_{Y_\sigma^{a}}\right),
\end{equation}
where the bold form $\mathbf h=\{h_\sigma\}_{\bar{\sigma}\geq \bar{\sigma}_0}$ and we assume throughout this section that $\bar\sigma_0$ is sufficiently
large. The localized subspace is
\begin{equation}\label{eq:Xim-def}
  \mathfrak Y_{\rm im}^{a}
  :=\left\{\mathbf f\in\mathfrak Y^{a}:\
  f_\sigma=\chi_{\rm im}f_\sigma\quad\hbox{for all }\bar\sigma\ge\bar\sigma_0\right\} .
\end{equation}
The choice of \(a\) is made after some exponential rates are fixed in this section.

For \(h\in Y_R\), let \(u^-_\sigma(h)\) solve the lower one-sided problem
\begin{equation}\label{eq:Tminus-problem}
  \mathcal L_\sigma u^-_\sigma(h)=0\quad\hbox{in }\Omega_R^-,\qquad
  \gamma_-u^-_\sigma(h)=h\quad\hbox{on }\Gamma_R,\qquad \partial_{\nu_\sigma}
  u_\sigma^-(h) = 0\quad\hbox{on\ }\Sigma_R\cap\Omega_R.
\end{equation}
Define
\begin{equation}\label{eq:Tminus-def}
  T^-_\sigma h:=\partial_{\nu_\sigma,-}u^-_\sigma(h)+M_\beta h\in Y_R' .
\end{equation}
For \(r\in Y_R'\), let \(v^+_\sigma(r)\) solve the upper one-sided Robin problem
\begin{equation}\label{eq:Rplus-problem}
  \mathcal L_\sigma v^+_\sigma(r)=0\quad\hbox{in }\Omega_R^+,\qquad
  \partial_{\nu_\sigma,+}v^+_\sigma(r)+M_\beta\gamma_+v^+_\sigma(r)=r
  \quad\hbox{on }\Gamma_R,\qquad \partial_{\nu_\sigma}
  v_\sigma^+(r) = 0\quad\hbox{on\ }\Sigma_R\cap\Omega_R,
\end{equation}
and set
\begin{equation}\label{eq:Rplus-def}
  R^+_\sigma r:=\gamma_+v^+_\sigma(r)\in Y_R .
\end{equation}
The composite map is
\begin{equation}\label{eq:Q-def}
  Q_\sigma:=R^+_\sigma T^-_\sigma .
\end{equation}
Below, we prove the exact inverse identity $(\mathcal A_{\sigma,N}^0)^{-1}= I - Q_\sigma$.

\begin{lemma}\label{lem:exact-inverse-identity}
Let \(f\in Y_R\), and let \(\phi_\sigma\) solve
  $\mathcal A_{\sigma,N}^0\phi_\sigma=f$.
Then
\begin{equation}\label{eq:identity-f-Qf}
  \phi_\sigma=f-Q_\sigma f .
\end{equation}
\end{lemma}

\begin{proof}
Set
\begin{equation}\label{eq:identity-layer-field}
  U_\sigma
  :=\mathrm{DL}_{\sigma,N}^0\phi_\sigma
   +\mathrm{SL}_{\sigma,N}^0(M_\beta\phi_\sigma).
\end{equation}
The jump relations give
\begin{align}
  \gamma_-U_\sigma
  &=\left(\frac12I+\mathcal K_{\sigma,N}^0+
      \mathcal S_{\sigma,N}^0M_\beta\right)\phi_\sigma
    =f,
    \label{eq:identity-minus-trace}\\
  \gamma_+U_\sigma
  &=\left(-\frac12I+\mathcal K_{\sigma,N}^0+
      \mathcal S_{\sigma,N}^0M_\beta\right)\phi_\sigma
    =f-\phi_\sigma .
    \label{eq:identity-plus-trace}
\end{align}
Let \(\eta_\sigma:=\gamma_+U_\sigma=f-\phi_\sigma\).  Since the conormal trace of the double-layer potential is continuous and the single-layer conormal trace has jump \(M_\beta\phi_\sigma\),
\begin{equation}\label{eq:identity-conormal-jump}
  \partial_{\nu_\sigma,+}U_\sigma-
  \partial_{\nu_\sigma,-}U_\sigma=M_\beta\phi_\sigma .
\end{equation}
Together with \(\gamma_-U_\sigma-\gamma_+U_\sigma=\phi_\sigma\), this gives
\begin{align}
 \left(\partial_{\nu_\sigma,-}U_\sigma+M_\beta\gamma_-U_\sigma\right)
 -\left(\partial_{\nu_\sigma,+}U_\sigma+M_\beta\gamma_+U_\sigma\right) 
 =-M_\beta\phi_\sigma+M_\beta(\gamma_-U_\sigma-\gamma_+U_\sigma)=0 .
 \label{eq:identity-robin-continuity}
\end{align}
Thus the upper trace \(\eta_\sigma\) is obtained by applying the upper Robin-to-Dirichlet map to the lower Robin trace generated by the Dirichlet datum \(f\):
\begin{equation}\label{eq:identity-eta-Q}
  \eta_\sigma=Q_\sigma f .
\end{equation}
Since \(\phi_\sigma=f-\eta_\sigma\), \eqref{eq:identity-f-Qf} follows.
\end{proof}

The operator \(Q_\sigma\) measures the return to the upper trace after one lower-side transmission.  It turns out that its leading part is the local half-line operator $Q_{\sigma,t}$  introduced below.  The local flat PML operator on the PML segment is
\begin{equation}\label{eq:flat-pml-equation}
  \mathcal L_{\sigma,t}u
  :=\frac1{\alpha_1}\partial_s^2u+
  \alpha_1\partial_n^2u+k^2\alpha_1u .
\end{equation}
The vertical artificial boundary is represented by \(\partial_su(0,n)=0\).  On \(n=0\) we use the same conormal convention as in the jump formula, namely \(\partial_{\nu_\sigma,-}w:=\alpha_1\partial_nw(s,0^-)\) and \(\partial_{\nu_\sigma,+}w:=\alpha_1\partial_nw(s,0^+)\).  On the flat part, \eqref{eq:beta-assump} gives \(M_\beta h=-Zh\).

For \(h\in H^{1/2}(0,\infty)\), let \(u_{\sigma,t}^-(h)\) be the downgoing solution of the lower half-line problem
\begin{equation}\label{eq:lower-flat-model}
\begin{cases}
 \mathcal L_{\sigma,t}u=0, & s>0,\ n<0,\\[1mm]
 u(s,0)=h(s), & s>0,\\[1mm]
 \partial_su(0,n)=0, & n<0,
\end{cases}
\end{equation}
and define the lower Dirichlet-to-Robin map by
\begin{equation}\label{eq:Tflat-def}
  T_{\sigma,t}^-h
  :=\partial_{\nu_\sigma,-}u_{\sigma,t}^-(h)+M_\beta h .
\end{equation}
For a Robin datum \(r\), let \(v_{\sigma,t}^+(r)\) be the upgoing solution of the upper half-line problem
\begin{equation}\label{eq:upper-flat-model}
\begin{cases}
 \mathcal L_{\sigma,t}v=0, & s>0,\ n>0,\\[1mm]
 \partial_{\nu_\sigma,+}v(s,0)+M_\beta v(s,0)=r(s), & s>0,\\[1mm]
 \partial_sv(0,n)=0, & n>0,
\end{cases}
\end{equation}
and define the upper Robin-to-Dirichlet map by
\begin{equation}\label{eq:Rflat-def}
  R_{\sigma,t}^+r:=v_{\sigma,t}^+(r)|_{n=0} .
\end{equation}
Set \(Q_{\sigma,t}:=R_{\sigma,t}^+T_{\sigma,t}^-\), which is the PML half-line analogue of \(Q_\sigma=R_\sigma^+T_\sigma^-\). It is used to control the leading part of \(Q_\sigma\mathcal B_\sigma^{\rm im}\) on \(\Gamma_{\rm im}\).

\begin{lemma}\label{lem:flat-symbol-ratio}
We have
\begin{equation}\label{eq:flat-composite-bound}
  \|Q_{\sigma,t}h\|_{H^{1/2}(0,\infty)}
  \le C\|h\|_{H^{1/2}(0,\infty)} .
\end{equation}
\end{lemma}

\begin{proof}
The Neumann boundary \(s=0\) is diagonalized by the cosine transform.  Let \(\mathcal C\) denote the cosine transform on \((0,\infty)\):
\begin{equation}\label{eq:cosine-transform}
  \widehat w_c(\xi)=\mathcal Cw(\xi)
  :=\sqrt{\frac2\pi}\int_0^\infty w(s)\cos(\xi s)\,ds,
  \qquad \xi\ge0 .
\end{equation}
By the cosine transform and separation of variables, with
\begin{equation}\label{eq:tau-symbol-def}
  \tau_\sigma(\xi):=\left(\xi^2-k^2\alpha_1^2\right)^{1/2},
  \qquad
  \tau_\sigma=A-\ii B,
  \quad A\ge0,
  \quad B\ge0,
\end{equation}
one obtains directly
\begin{equation}\label{eq:Qflat-symbol}
  \widehat{Q_{\sigma,t}h}_c(\xi)
  =\frac{\tau_\sigma(\xi)-Z}{-\tau_\sigma(\xi)-Z}\widehat h(\xi).
\end{equation}
Basic algebraic manipulations show that, 
\begin{equation}\label{eq:flat-ratio}
  \left|\frac{\tau_\sigma(\xi)-Z}{-\tau_\sigma(\xi)-Z}\right|
  +\left|\frac{-\tau_\sigma(\xi)-Z}{\tau_\sigma(\xi)-Z}\right|
  \le C .
\end{equation}
Finally, \eqref{eq:flat-composite-bound} follows from \eqref{eq:Qflat-symbol},
\end{proof}


\subsection{A nearest-image corrected BIE}\label{subsec:nearest-image-computable}

We first study an intermediate computable equation in which the nearest Neumann
reflected image is retained explicitly.  

We define the nearest-image kernel using the reflected series in
Section~\ref{sec:exact-green}.  Decompose \(\Gamma_{\rm im}=\Gamma_{\rm
im}^{L}\cup\Gamma_{\rm im}^{R}\) and \(\chi_{\rm im}=\chi_{\rm im}^{L}+\chi_{\rm
im}^{R}\), where the superscripts indicate the PML segments adjacent to the left
and right vertical Neumann boundaries.  On each component \(\Gamma_{\rm
im}^{\iota}\), \(\iota=L,R\), we use the fixed local coordinate system $(s,n)$
in Section~\ref{sec:weighted-exact-stability}, with the adjacent boundary placed
at \(s=0\), and write \(X_{\rm T}^{\iota}(s,n)=(\alpha_1s,n)\) and
\((s,n)^{\iota,*}=(-s,n)\).  The nearest-image kernel consists of the two fixed
boundary reflections
\begin{equation}\label{eq:G-near-def}
  G_\sigma^{\rm near}(x,y)
  :=
  \begin{cases}
  \Phi_k\big(X_{\rm T}^{\iota}(x),X_{\rm T}^{\iota}(y^{\iota,*})\big),
  & x,y\in\Gamma_{\rm im}^{\iota},\quad \iota=L,R,\\
  0,& x\in\Gamma_{\rm im}^{L},\ y\in\Gamma_{\rm im}^{R}
      \hbox{ or }x\in\Gamma_{\rm im}^{R},\ y\in\Gamma_{\rm im}^{L}.
  \end{cases}
\end{equation}
Thus no left--right cross reflection is included in this kernel.  Equivalently, the following operators are defined componentwise:
\begin{align}
  \mathcal S_{\sigma}^{\rm im}q
  &:=\chi_{\rm im} \int_{\Gamma_R}G_\sigma^{\rm near}(\cdot,y)
      \chi_{\rm im}(y)q(y)\,ds_y,
  \label{eq:S-im-def}\\
  \mathcal K_{\sigma}^{\rm im}\phi
  &:=\chi_{\rm im}\operatorname{p.v.}\int_{\Gamma_R}
      \partial_{\nu_\sigma,y}^{*}G_\sigma^{\rm near}(\cdot,y)
      \chi_{\rm im}(y)\phi(y)\,ds_y,
  \label{eq:K-im-def}\\
  \mathcal B_{\sigma}^{\rm im}
  &:=\mathcal K_{\sigma}^{\rm im}+\mathcal S_{\sigma}^{\rm im}M_\beta
    =\mathcal S_{\sigma}^{\rm im}M_\beta .
  \label{eq:B-im-def}
\end{align}
Here we use the fact that the double-layer part ${\cal K}_\sigma^{\rm im}=0$,
since $\Gamma_{\rm im}$ is flat; we keep it here to make the following decomposition more
clearly. An important observation is 
\begin{equation}\label{eq:Bim-localized}
  \mathcal B_{\sigma}^{\rm im}
  =\chi_{\rm im}\mathcal B_{\sigma}^{\rm im}\chi_{\rm im} .
\end{equation}
Consequently, $\mathcal A_{\sigma,N}^0$ can be decomposed into three parts as follows:
\begin{equation}\label{eq:A0-Atilde-decomposition}
  \mathcal A_{\sigma,N}^0=\mathcal A_\sigma^c + \mathcal B_{\sigma}^{\rm im}+\mathcal E_\sigma.
\end{equation}
Here, $\mathcal A_\sigma^c$ contains the stretched free-space kernel $\Phi_\sigma$. For $\mathcal B_{\sigma}^{\rm im}$, the kernel in each PML segment $\Gamma_{\rm im}^{\iota}$ has the local form \(\Phi_k\big((\alpha_1s,n),(-\alpha_1s',n')\big)\), and hence depends on the reflected tangential distance \(s+s'\).  This is precisely the nearest vertical-boundary reflected term in the image expansion \eqref{eq:neumann-image-series}. The remaining part \(\mathcal E_\sigma\) includes the left--right, right--left, and all other reflected terms that are separated from this nearest-image configuration. The following lemma justfies that ${\cal E}_\sigma$ provides an exponentially small perturbation.

\begin{lemma}\label{lem:Ainv-E-small}
There exists \(c_1>0\), independent of \(\bar\sigma_0\), such that
\begin{equation}\label{eq:Ainv-E-small}
  \left\|(\mathcal A_{\sigma,N}^0)^{-1}\mathcal E_\sigma\right\|_{\mathfrak Y^{a}\to\mathfrak Y^{a}}
  \le \mathcal P(\bar\sigma_0)e^{-(c_1-a)\bar\sigma_0},
  \qquad 0<a<c_1 .
\end{equation}
\end{lemma}

\begin{proof}
By Proposition~\ref{prop:reflected-kernel-estimate}, or equivalently by the
differentiated Hankel bound \eqref{eq:Greendecay} followed by the standard
local-coordinate Sobolev mapping estimate, for \(r_\sigma=\mathcal E_\sigma
h_\sigma\),
\begin{align}
  \|r_\sigma\|_{Y_R}
  &\le \mathcal P(\bar\sigma)e^{-c_1\bar\sigma}\|h_\sigma\|_{Y_\sigma^{a}},
  \label{eq:E-r-unweighted-small}\\
  W_\sigma\|\chi_{\rm im}r_\sigma\|_{Y_R}
  &\le \mathcal P(\bar\sigma)e^{-(c_1-a)\bar\sigma}\|h_\sigma\|_{Y_\sigma^{a}} .
  \label{eq:E-r-small}
\end{align}
To estimate \(Q_\sigma r_\sigma\),  we directly apply Theorem~\ref{thm:pml-pde} and Lemma~\ref{lem:aux-mixed-unique}  to the two one-sided subproblems \eqref{eq:lower-flat-model}  and \eqref{eq:upper-flat-model} defining \(Q_\sigma\), yielding
\begin{equation}\label{eq:QE-unweighted-poly}
  \|Q_\sigma r_\sigma\|_{Y_R}
  \le \mathcal P(\bar\sigma)\|r_\sigma\|_{Y_R} .
\end{equation}
Therefore
\begin{align}\label{eq:QE-r-small}
  \|Q_\sigma r_\sigma\|_{Y_R}
   +W_\sigma\|\chi_{\rm im}Q_\sigma r_\sigma\|_{Y_R}       
  \le
  \mathcal P(\bar\sigma)(1+W_\sigma)e^{-c_1\bar\sigma}
   \|h_\sigma\|_{Y_\sigma^{a}}          
  \le \mathcal P(\bar\sigma)e^{-(c_1-a)\bar\sigma}
   \|h_\sigma\|_{Y_\sigma^{a}}.
\end{align}
Since \((\mathcal A_{\sigma,N}^0)^{-1}r_\sigma=r_\sigma-Q_\sigma r_\sigma\), \eqref{eq:E-r-small}--\eqref{eq:QE-r-small} yield the fixed-\(\bar\sigma\) estimate.  Taking \(\sup_{\bar\sigma\ge\bar\sigma_0}\) proves \eqref{eq:Ainv-E-small}.
\end{proof}

Let \(\mathcal A_\sigma^{\rm ni}:=\mathcal A_\sigma^c+\mathcal B_\sigma^{\rm im}\). 
Consider the nearest-image corrected equation
\begin{equation}\label{eq:nearest-image-bie}
  \mathcal A_\sigma^{\rm ni}\phi_\sigma^{\rm ni}
  =\mathcal S_\sigma^c g .
\end{equation}
For \(x\) in the physical region $\Omega_p$, define the corresponding nearest-image reconstructed field by
\begin{align}
  u_\sigma^{\rm ni}(x)
  &:=(\mathrm{SL}_\sigma^c+\mathrm{SL}_\sigma^{\rm im})
       (g-M_\beta\phi_\sigma^{\rm ni})(x)
     -(\mathrm{DL}_\sigma^c+\mathrm{DL}_\sigma^{\rm im})
       \phi_\sigma^{\rm ni}(x),
  \label{eq:nearest-image-field}
\end{align}
where \(\mathrm{SL}_\sigma^{\rm im}\) and \(\mathrm{DL}_\sigma^{\rm im}\) are the volume layer potentials generated by the same localized nearest-image kernel used in
\eqref{eq:S-im-def}--\eqref{eq:K-im-def}.  Since \(g\) is supported in $\Omega_p$, the image single-layer source term vanishes when applied to \(g\); it is kept in \eqref{eq:nearest-image-field} only to show a form compared with the exact one \eqref{eq:green-rep-pde}. The next theorem illustrates the exponential convergence of $u_\sigma^{\rm ni}$ to $u$ in the physical domain $\Omega_p$.
\begin{theorem}\label{thm:nearest-image-corrected-conv}
The nearest-image corrected equation \eqref{eq:nearest-image-bie} is uniquely solvable.  Moreover, for any compact set \(K\Subset\Omega_p\), the solution \eqref{eq:nearest-image-field} satisfies
\begin{equation}\label{eq:nearest-image-total-field-error}
  \|u-u_\sigma^{\rm ni}\|_{H^1(K)}
  \le \mathcal P(\bar\sigma)\left(e^{-c_{\rm ni}\bar\sigma}
       +e^{-Z(L_2+d_2)}\right)\|g\|_{Y_R'} .
\end{equation}
\end{theorem}

\begin{proof}
Since \(\mathcal A_{\sigma,N}^0-\mathcal A_\sigma^{\rm ni}=\mathcal E_\sigma\), subtracting
\(\mathcal A_{\sigma,N}^0\phi_\sigma^0=\mathcal S_{\sigma,N}^0g\) from
\eqref{eq:nearest-image-bie} gives, with
\(e_\sigma^{\rm ni}:=\phi_\sigma^{\rm ni}-\phi_\sigma^0\),
\[
  e_\sigma^{\rm ni}
  =(\mathcal A_{\sigma,N}^0)^{-1}\mathcal E_\sigma e_\sigma^{\rm ni}
  +(\mathcal A_{\sigma,N}^0)^{-1}
     \left\{(\mathcal S_\sigma^c-\mathcal S_{\sigma,N}^0)g
     +\mathcal E_\sigma\phi_\sigma^0\right\} .
\]
By Lemma~\ref{lem:Ainv-E-small}, the first operator on the right-hand side has norm smaller than \(1/2\) for large \(\bar\sigma\).  Hence \eqref{eq:nearest-image-bie} is uniquely solvable by the Neumann series.  The same estimate as in Lemma~\ref{lem:Ainv-E-small}, together with \eqref{eq:exact-physical-unweighted}, gives
\[
  \|e_\sigma^{\rm ni}\|_{Y_R}
  \le \mathcal P(\bar\sigma)e^{-c\bar\sigma}\|g\|_{Y_R'} .
\]
Comparing \eqref{eq:nearest-image-field} with \eqref{eq:exact-reconstruction}, and then applying Proposition~\ref{prop:reflected-kernel-estimate}, 
\[
  \|u_\sigma^{\rm ni}-u_\sigma\|_{H^1(K)}
  \le \mathcal P(\bar\sigma)e^{-c\bar\sigma}\|g\|_{Y_R'} .
\]
Combining this estimate with Theorem~\ref{thm:exact-conv}  proves \eqref{eq:nearest-image-total-field-error}.
\end{proof}

From Theorem~\ref{thm:nearest-image-corrected-conv}, the nearest-image
corrected equation \(\mathcal A_\sigma^{\rm ni}\phi_\sigma^{\rm ni}=\mathcal
S_\sigma^c g\) already provides a solution that exponentially converges to the
true solution $u$ in the physical region $\Omega_p$. Compared with the computable PML-BIE
\eqref{eq:computable-bie2}, the extra kernel in ${\cal B}_\sigma^{\rm im}$ is
equally simple, introduces no complicated integrands, and requires no additional numerical effort. Thus the analysis could stop here if one were willing to use the nearest-image corrected BIE \eqref{eq:nearest-image-bie}. Nevertheless, this intermediate equation contains a finite-boundary image
term and therefore is not a direct truncation of the associated complex-scaled BIE in Section~\ref{sec:cs-bie}.  For this reason we keep going to remove ${\cal B}_\sigma^{\rm im}$.  By Lemma~\ref{lem:exact-inverse-identity}, 
\[
(\mathcal
A_{\sigma,N}^0)^{-1}\mathcal B_\sigma^{\rm im}
= \mathcal B_\sigma^{\rm im} - Q_\sigma  \mathcal B_\sigma^{\rm im}.
\] 
The next two subsections are devoted to estimate $ \mathcal B_\sigma^{\rm im} $ and $ Q_\sigma  \mathcal B_\sigma^{\rm im}$ separately.

\subsection{\texorpdfstring{Estimate of ${\cal B}_\sigma^{\rm im}$}{}}\label{subsec:computable-splitting}

We now return to the original stretched-kernel equation
\eqref{eq:computable-bie2}. Before proceeding, we prove the following technical lemma first.
\begin{lemma}\label{lem:endpoint-Hankel}
Let
\begin{equation}\label{eq:Hankel-endpoint-operator}
  (\mathcal H_\mu q)(s):=\theta(s)\int_0^\ell H_0^{(1)}(\mu(s+s'))\theta(s')q(s')\,ds',
  \qquad \mu=k\alpha_1,
\end{equation}
where \(\ell>0\) is fixed, \(\theta\in C_0^\infty([0,\ell))\) is fixed, and \(|\mu|\eqsim \Im\mu\eqsim \bar\sigma\).  Then
\begin{equation}\label{eq:endpoint-Hankel-Hhalf}
  \|\mathcal H_\mu q\|_{H^{1/2}(0,\ell)}
  \le C\bar\sigma^{-1/2}\|q\|_{H^{1/2}(0,\ell)} .
\end{equation}
\end{lemma}

\begin{proof}
Put $M:=|\mu|$,$\Im\mu\ge cM$, and $K(t):=H_0^{(1)}(\mu t).$
We first estimate the kernel in \(L^2((0,\ell)^2)\).  The standard Hankel bounds in the sector \(\Im(\mu t)\ge cMt\) give
\begin{align}
  |H_0^{(1)}(\mu t)|&\le C(1+|\log(Mt)|), &&0<t\le M^{-1},
    \label{eq:Hankel-small-bound}\\
  |H_0^{(1)}(\mu t)|&\le C(Mt)^{-1/2}e^{-cMt}, &&t\ge M^{-1}.
    \label{eq:Hankel-large-bound}
\end{align}
Since \(s+s'=t\) has one-dimensional measure at most \(t\),
\begin{align}
  \int_0^\ell\int_0^\ell |K(s+s')|^2\,ds'\,ds
  &\le C\int_0^{2\ell} t |H_0^{(1)}(\mu t)|^2\,dt                       \notag\\
  &\le CM^{-2}\int_0^2 \rho(1+|\log\rho|)^2\,d\rho
      +CM^{-2}\int_1^\infty e^{-2c\rho}\,d\rho \le CM^{-2} .
  \label{eq:Hankel-HS-direct}
\end{align}
Hence
\begin{equation}\label{eq:Hankel-L2-L2}
  \|\mathcal H_\mu q\|_{L^2(0,\ell)}
  \le CM^{-1}\|q\|_{L^2(0,\ell)} .
\end{equation}

It remains to obtain a uniform \(H^1\)-bound.  Set \(a(s'):=\theta(s')q(s')\).  For \(q\in H^1(0,\ell)\),
\begin{align}
  \partial_s(\mathcal H_\mu q)(s)
  &=\theta'(s)\int_0^\ell K(s+s')a(s')\,ds'
    +\theta(s)\int_0^\ell K'(s+s')a(s')\,ds'.
  \label{eq:Hankel-derivative-start}
\end{align}
Integration by parts,
\begin{equation}\label{eq:Hankel-int-by-parts}
  \int_0^\ell K'(s+s')a(s')\,ds'
  =-K(s)a(0)-\int_0^\ell K(s+s')a'(s')\,ds'.
\end{equation}
The two integral terms in \eqref{eq:Hankel-derivative-start}--\eqref{eq:Hankel-int-by-parts} are estimated by \eqref{eq:Hankel-L2-L2}.
For the endpoint term, \eqref{eq:Hankel-small-bound}--\eqref{eq:Hankel-large-bound} imply
\begin{align}
  \|K(\cdot)a(0)\|_{L^2(0,\ell)}^2
  \le C|a(0)|^2\int_0^\ell |H_0^{(1)}(\mu s)|^2\,ds
  \le CM^{-1}|a(0)|^2
  \le C\|q\|_{H^1(0,\ell)}^2 .
  \label{eq:Hankel-boundary-term}
\end{align}
Combining \eqref{eq:Hankel-derivative-start}--\eqref{eq:Hankel-boundary-term},
\begin{equation}\label{eq:Hankel-H1H1}
  \|\mathcal H_\mu q\|_{H^1(0,\ell)}
  \le C\|q\|_{H^1(0,\ell)} .
\end{equation}
Since \(M\eqsim \bar\sigma\), interpolation of \eqref{eq:Hankel-L2-L2} and \eqref{eq:Hankel-H1H1} gives
\eqref{eq:endpoint-Hankel-Hhalf}.
\end{proof}

We show the algebraic smallness of ${\cal B}_\sigma$ below.
\begin{lemma}\label{lem:Bim-unweighted}\label{lem:Bim-family}
The operator ${\cal B}_\sigma$ satisfies
\begin{equation}\label{eq:Bim-unweighted}
  \|\mathcal B_{\sigma}^{\rm im}\|_{Y_R\to Y_R}
  \le C\bar\sigma^{-1/2} .
\end{equation}
Moreover, for any \(0<a<c_0\),
\begin{equation}\label{eq:Bim-family-bound}
  \|\mathcal B_\sigma^{\rm im}\mathbf h\|_{\mathfrak Y^{a}}
  \le C\bar\sigma_0^{-1/2}\|\mathbf h\|_{\mathfrak Y^{a}},
  \qquad
  \mathcal B_\sigma^{\rm im}\mathbf h\in\mathfrak Y_{\rm im}^{a} .
\end{equation}
\end{lemma}

\begin{proof}
In one PML segment, \(s,s'\ge0\) denote distances from the vertical boundary along the flat boundary.  The reflected tangential distance is
\begin{equation}\label{eq:nearest-image-distance}
  \zeta_\sigma(s,s')=\alpha_1(s+s') .
\end{equation}
Thus the kernel of ${\cal S}_\sigma^{\rm im}$ is 
\begin{equation}\label{eq:single-image-kernel}
  K^S_\sigma(s,s')=\chi_{\rm im}(s)\chi_{\rm im}(s')H_0^{(1)}(k\alpha_1(s+s')) .
\end{equation}
Applying Lemma~\ref{lem:endpoint-Hankel} with \(\theta=\chi_{\rm im}\) and with \(\ell\) chosen so that \( [0,\ell)\) contains \(\operatorname{supp}\chi_{\rm im}\),
\begin{equation}\label{eq:single-image-Hhalf}
  \|\mathcal S_{\sigma}^{\rm im}q\|_{H^{1/2}}
  \le C\bar\sigma^{-1/2}\|q\|_{H^{1/2}} .
\end{equation}
This together with \eqref{eq:Bim-localized} implies
\begin{equation}\label{eq:Bim-final-unweighted}
  \|\mathcal B_{\sigma}^{\rm im}\phi\|_{Y_R}
  =\|\mathcal S_{\sigma}^{\rm im}M_\beta\chi_{\rm im}\phi\|_{Y_R}
  \le C\bar\sigma^{-1/2}\|\chi_{\rm im}\phi\|_{Y_R}
  \le C\bar\sigma^{-1/2}\|\phi\|_{Y_R} .
\end{equation}
This proves \eqref{eq:Bim-unweighted}.  The inclusion in \(\mathfrak Y_{\rm im}^{a}\) follows from the sharper localized form of \eqref{eq:Bim-final-unweighted}:
\begin{align}
  W_\sigma\|\chi_{\rm im}\mathcal B_{\sigma}^{\rm im}h_\sigma\|_{Y_R}
  =W_\sigma\|\mathcal B_{\sigma}^{\rm im}\chi_{\rm im}h_\sigma\|_{Y_R} 
  \le C\bar\sigma^{-1/2}W_\sigma\|\chi_{\rm im}h_\sigma\|_{Y_R} .
  \label{eq:Bim-weight-cancel}
\end{align}
Taking \(\sup_{\bar\sigma\ge\bar\sigma_0}\) proves \eqref{eq:Bim-family-bound}.
\end{proof}

\subsection{\texorpdfstring{Estimate of $Q_\sigma\mathcal B_{\sigma}^{\rm im}$}{}}\label{subsec:QBim-estimate}

Suppose \(\mathbf h\in\mathfrak Y^{a}\) and we work for each fixed \(\bar\sigma\) in the following.  Set
\begin{equation}\label{eq:QBim-qf-def}
  q_\sigma:=\chi_{\rm im}h_\sigma,
  \qquad
  f_\sigma:=\mathcal B_\sigma^{\rm im}h_\sigma .
\end{equation}
By \eqref{eq:Bim-localized} and Lemma~\ref{lem:Bim-unweighted},
\begin{equation}\label{eq:QBim-f-size}
  f_\sigma=\chi_{\rm im}f_\sigma,
  \qquad
  \|f_\sigma\|_{Y_R}\le C\bar\sigma^{-1/2}\|q_\sigma\|_{Y_R} .
\end{equation}

We next fix the separation geometry used for the weighted estimate.  Choose a
smooth function \(\eta=\eta(s)\) in the local coordinate $s$ such that $\eta=1$
on a neighborhood of supp $\chi_{\rm im}$, and that supp $\eta$ lies within
$\Sigma_{\rm pml}$.  We choose it so that \(\eta'\) is supported away from the
vertical boundaries and from \(\operatorname{supp}\chi_{\rm im}\).  Set
\begin{equation}\label{eq:QBim-separated-region}
  \mathcal C_b^\pm:=\Omega_R^\pm\cap\operatorname{supp}(\eta') ,
  \qquad
  d_b:=\operatorname{dist}_s(\operatorname{supp}\chi_{\rm im},\operatorname{supp}\eta')>0 .
\end{equation}
Since \(\operatorname{supp}\eta'\) lies in $\Sigma_{\rm pml}$, crossing from
\(\operatorname{supp}\chi_{\rm im}\) to \(\mathcal C_b^\pm\) gives a fixed
stretched distance.  Let
\begin{equation}\label{eq:QBim-flat-pair}
  \mathbf U_{\sigma,t}:=(U_{\sigma,t}^-,U_{\sigma,t}^+)
  :=\big(u_{\sigma,t}^-(f_\sigma),\,v_{\sigma,t}^+(T_{\sigma,t}^-f_\sigma)\big).
\end{equation}
Then $Q_{\sigma,t}f_\sigma=\gamma_+U_{\sigma,t}^+$.
The following two lemmas estimate the local difference between $Q_\sigma$ and $Q_{\sigma,t}$.
\begin{lemma}\label{lem:QBim-correction-separated}
The correction operator $Q_\sigma-Q_{\sigma,t}$ satisfies
\begin{equation}\label{eq:QBim-correction-by-U}
  \|\chi_{\rm im}(Q_\sigma-Q_{\sigma,t})f_\sigma\|_{Y_R}
  \le \mathcal P(\bar\sigma)
       \sum_\pm\|U_{\sigma,t}^\pm\|_{H^1(\mathcal C_b^\pm)} .
\end{equation}
\end{lemma}

\begin{proof}
Let \(\mathbf U_\sigma=(U_\sigma^-,U_\sigma^+)\) denote the exact pair defining \(Q_\sigma f_\sigma\):
\begin{subequations}\label{eq:QBim-exact-pair}
\begin{align}
  \mathcal L_\sigma U_\sigma^-&=0,
  &\gamma_-U_\sigma^-&=f_\sigma,
  \label{eq:QBim-exact-lower}\\
  \mathcal L_\sigma U_\sigma^+&=0,
  &\partial_{\nu_\sigma,+}U_\sigma^+ +M_\beta\gamma_+U_\sigma^+
   &=\partial_{\nu_\sigma,-}U_\sigma^-+M_\beta f_\sigma .
  \label{eq:QBim-exact-upper}
\end{align}
\end{subequations}
Both fields satisfy the homogeneous Neumann condition on the artificial boundary.  Define
\begin{equation}\label{eq:QBim-W-def}
  \mathbf W_{\sigma,b}:=(W_{\sigma,b}^-,W_{\sigma,b}^+)
  :=\mathbf U_\sigma-\eta\mathbf U_{\sigma,t} .
\end{equation}
On \(\operatorname{supp}\chi_{\rm im}\), \(\eta=1\); therefore
\begin{equation}\label{eq:QBim-correction-identity}
  \chi_{\rm im}(Q_\sigma-Q_{\sigma,t})f_\sigma
  =\chi_{\rm im}\gamma_+W_{\sigma,b}^+ .
\end{equation}
We now compute the equation for \(\mathbf W_{\sigma,b}\).  In \(\operatorname{supp}\eta'\) the coefficients are the constant PML coefficients, and
\begin{equation}\label{eq:QBim-commutator}
  \mathcal L_\sigma(\eta U_{\sigma,t}^\pm)
  =\eta\mathcal L_\sigma U_{\sigma,t}^\pm
   +\alpha_1^{-1}\big(\eta''U_{\sigma,t}^\pm
          +2\eta'\partial_sU_{\sigma,t}^\pm\big)
  =\alpha_1^{-1}\big(\eta''U_{\sigma,t}^\pm
          +2\eta'\partial_sU_{\sigma,t}^\pm\big).
\end{equation}
Thus, with zero extension outside \(\mathcal C_b^\pm\), set
\begin{equation}\label{eq:QBim-F-def}
  F_{\sigma,b}^\pm
  :=-\alpha_1^{-1}\big(\eta''U_{\sigma,t}^\pm
          +2\eta'\partial_sU_{\sigma,t}^\pm\big).
\end{equation}
Then \(\mathbf W_{\sigma,b}\) satisfies
\begin{subequations}\label{eq:QBim-W-PDE}
\begin{align}
  \mathcal L_\sigma W_{\sigma,b}^- &=F_{\sigma,b}^- &&\hbox{in }\Omega_R^- ,
  \label{eq:QBim-W-lower-pde}\\
  \gamma_-W_{\sigma,b}^- &=0 &&\hbox{on }\Gamma_R ,
  \label{eq:QBim-W-lower-bc}\\
  \mathcal L_\sigma W_{\sigma,b}^+ &=F_{\sigma,b}^+ &&\hbox{in }\Omega_R,
  \label{eq:QBim-W-upper-pde}\\
  \partial_{\nu_\sigma,+}W_{\sigma,b}^+ +M_\beta\gamma_+W_{\sigma,b}^+
  &=\partial_{\nu_\sigma,-}W_{\sigma,b}^- &&\hbox{on }\Gamma_R .
  \label{eq:QBim-W-upper-bc}
\end{align}
\end{subequations}
We verify the two boundary conditions \eqref{eq:QBim-W-lower-bc} and \eqref{eq:QBim-W-upper-bc}. On one hand, \(\gamma_-W_{\sigma,b}^-=f_\sigma-\eta f_\sigma=0\).  On the other hand, since \(\eta\) depends only on the tangential coordinate $s$ and is constant in supp $\chi_{\rm im}$,
\begin{align}
\partial_{\nu_\sigma,+}W_{\sigma,b}^+ +M_\beta\gamma_+W_{\sigma,b}^+ &\quad=\partial_{\nu_\sigma,-}U_\sigma^-+M_\beta f_\sigma
 -\eta\big(\partial_{\nu_\sigma,+}U_{\sigma,t}^+
      +M_\beta\gamma_+U_{\sigma,t}^+\big)                                  \notag\\
&\quad=\partial_{\nu_\sigma,-}U_\sigma^-+M_\beta f_\sigma
 -\eta\big(\partial_{\nu_\sigma,-}U_{\sigma,t}^-+M_\beta f_\sigma\big)       \notag\\
&\quad=\partial_{\nu_\sigma,-}(U_\sigma^- -\eta U_{\sigma,t}^-)
 +(1-\eta)M_\beta f_\sigma
  =\partial_{\nu_\sigma,-}W_{\sigma,b}^- .
  \label{eq:QBim-coupling-check}
\end{align}
The last equality uses \((1-\eta)f_\sigma=0\).

By \eqref{eq:QBim-F-def}, 
\begin{equation}\label{eq:QBim-F-by-U}
  \sum_\pm\|F_{\sigma,b}^\pm\|_{H^{-1}_{0}(\Omega_R^\pm)}
  \le \mathcal P(\bar\sigma)
       \sum_\pm\|U_{\sigma,t}^\pm\|_{H^1(\mathcal C_b^\pm)} .
\end{equation}
Applying the forced mixed estimate \eqref{eq:aux-mixed-bound} to the lower
subproblem \eqref{eq:QBim-W-lower-pde}--\eqref{eq:QBim-W-lower-bc}, and then the
forced Robin estimate \eqref{eq:forced-robin-pml-est} to the upper subproblem
\eqref{eq:QBim-W-upper-pde}--\eqref{eq:QBim-W-upper-bc}, we obtain from
\eqref{eq:QBim-correction-identity} that
\begin{align}
  \|\chi_{\rm im}(Q_\sigma-Q_{\sigma,t})f_\sigma\|_{Y_R}
  \le \|W_{\sigma,b}^+\|_{H^1(\Omega_R^+)}\le \mathcal P(\bar\sigma)
       \sum_\pm\|U_{\sigma,t}^\pm\|_{H^1(\mathcal C_b^\pm)}.
\end{align}
\end{proof}

\begin{lemma}\label{lem:QBim-separated-fields}
The flat pair $\mathbf U_{\sigma,t}$ satisfies
\begin{equation}\label{eq:QBim-separated-U-small}
  \sum_\pm\|U_{\sigma,t}^\pm\|_{H^1(\mathcal C_b^\pm)}
  \le \mathcal P(\bar\sigma) e^{-c_b\bar\sigma}\|q_\sigma\|_{Y_R} .
\end{equation}
\end{lemma}

\begin{proof}
Since
\[
  f_\sigma(s)=\int_{\operatorname{supp}\chi_{\rm im}}
  B_\sigma^{\rm im}(s,s')q_\sigma(s')\,ds',\qquad
  B_\sigma^{\rm im}(s,s')=(-Z)\frac{\ii}{4}\chi_{\rm im}(s)\chi_{\rm im}(s')
  H_0^{(1)}(k\alpha_1(s+s')),
\]
it is enough to estimate the response generated by one boundary profile
\(b_{\sigma,s'}(s):=B_\sigma^{\rm im}(s,s')\), and then integrate the resulting kernel with \(q_\sigma(s')\).

We first derive the two Green kernels $D_{\sigma,\alpha}^\pm (x,s_0)$ excited by
a boundary source $\delta(s_0)$ at the point $(s_0,0)$.  For local points \(x=(s,n)\),
put \(X_t(x)=(\alpha_1s,n)\) and
\(\Phi_{\sigma,t}(x,y)=\frac{\ii}{4}H_0^{(1)}(k|X_t(x)-X_t(y)|)\).  The Neumann
boundary \(s=0\) is imposed by even reflection.  Thus, with \(y^v=(-s_y,n_y)\),
\(y^h=(s_y,-n_y)\), and \(y^{vh}=(-s_y,-n_y)\), the lower Dirichlet Green function for Problem \eqref{eq:lower-flat-model} is obtained from
\[
  G_{\sigma,DN}^-(x,y)=\Phi_{\sigma,t}(x,y)+\Phi_{\sigma,t}(x,y^v)-\Phi_{\sigma,t}(x,y^h)-\Phi_{\sigma,t}(x,y^{vh}),
\]
namely
\[
  D_{\sigma,\alpha}^-(x,s_0):=-\partial_x^\alpha\partial_{\nu_{\sigma,y},-}G_{\sigma,DN}^-(x,(s_0,0)),\qquad |\alpha|\le1.
\]
Thus the lower response generated by \(b_{\sigma,s'}\) has kernel
\[
  K_{\sigma,\alpha}^{-,b}(x,s')=\int_0^\infty D_{\sigma,\alpha}^-(x,s_0)B_\sigma^{\rm im}(s_0,s')\,ds_0,
  \qquad x\in\mathcal C_b^- .
\]
The corresponding Robin datum for the upper problem is
\[
  r_{\sigma,s'}^-(s_1):=M_\beta B_\sigma^{\rm im}(s_1,s')
  -\partial_{\nu_{\sigma,x},-}\int_0^\infty
  \partial_{\nu_{\sigma,y},-}G_{\sigma,DN}^-((s_1,0),(s_0,0))B_\sigma^{\rm im}(s_0,s')\,ds_0,
\]

For the upper Robin problem \eqref{eq:upper-flat-model}, let \(G_{\sigma,R}^{\rm hp}\) be the complex-stretched flat Robin half-plane Green function of \cite[Secs.~2 \& 4]{JiangLi2020}.  The Neumann-boundary Robin Green function is
\[
  G_{\sigma,RN}^+(x,s_1):=-\lim_{\varepsilon\downarrow0}
  \big(G_{\sigma,R}^{\rm hp}(x,(s_1,\varepsilon))+G_{\sigma,R}^{\rm hp}(x,(-s_1,\varepsilon))\big),\qquad s_1>0.
\]
Thus, $D_{\sigma,\alpha}^+(x,s_0) = \partial_x^\alpha G_{\sigma,RN}^+(x,s_0)$ and the upper response has the kernel
\[
  K_{\sigma,\alpha}^{+,b}(x,s')=\int_{0}^{\infty} D_{\sigma,\alpha}^+(x,s_0) r_{\sigma,s'}^-(s_0)ds_0,
  \qquad x\in\mathcal C_b^+ .
\]

We now estimate the two response kernels.  All constants in the following estimates are uniform for the observation point \(x\) in the corresponding set \(\mathcal C_b^-\) or \(\mathcal C_b^+\). Since \(\mathcal C_b^-\) is separated from \(\operatorname{supp}\chi_{\rm im}\), the differentiated Hankel bound \eqref{eq:Greendecay}, applied to the four-image formula for \(G_{\sigma,DN}^-\), gives
\[
  |D_{\sigma,\alpha}^-(x,s_0)|\le \mathcal P(\bar\sigma)e^{-c_b\bar\sigma},
  \qquad x\in\mathcal C_b^-,\quad s_0\in\operatorname{supp}\chi_{\rm im},\quad |\alpha|\le1.
\]
Together with the endpoint Hankel bound for \(B_\sigma^{\rm im}\), this yields
\[
  |K_{\sigma,\alpha}^{-,b}(x,s')|\le \mathcal P(\bar\sigma)e^{-c_b\bar\sigma},
  \qquad x\in\mathcal C_b^-,\quad s'\in\operatorname{supp}\chi_{\rm im},\quad |\alpha|\le1.
\]
For the upper response, the Robin Green estimate of
\cite[Lemma~4.4]{JiangLi2020} applies to \(G_{\sigma,RN}^+\).  Unlike the lower Dirichlet input, the
Robin datum \(r_{\sigma,s'}^-\) on the common flat boundary is not compactly
supported; it is exponentially small away from \(\operatorname{supp}\chi_{\rm im}\). For a fixed
\(x=(s_x,n_x)\in\mathcal C_b^+\), split \((0,\infty)=I_\chi\cup I_x\cup I_0\),
where \(I_\chi\) is a small neighbourhood of \(\operatorname{supp}\chi_{\rm
im}\), \(I_x\) is a small neighbourhood of \(s_x\), and the two neighbourhoods
are disjoint.  On \(I_\chi\) the Robin Green trace is  smooth and exponentially small so that the hypersingularity for $r_{\sigma,s'}^-$ can be treated by one integration by parts; on
\(I_0\) the datum \(r_{\sigma,s'}^-\) is exponentially small; on \(I_x\), the
datum \(r_{\sigma,s'}^-\) is smooth and exponentially small because \(I_x\) is
separated from \(\operatorname{supp}\chi_{\rm im}\), and the tangential
derivative is again treated by one integration by parts.  Hence
\[
  |K_{\sigma,\alpha}^{+,b}(x,s')|\le \mathcal P(\bar\sigma)e^{-c_b\bar\sigma},
  \qquad x\in\mathcal C_b^+,\quad
  s'\in\operatorname{supp}\chi_{\rm im},\quad |\alpha|\le1 .
\]
The polynomial $\mathcal P$ is independent of \(x\in\mathcal
C_b^+\), because the separation distance used in \(I_\chi\), the exponential
smallness on \(I_0\cup I_x\), and the local integration-by-parts estimate on
\(I_x\) are all uniform for \(x\in\mathcal C_b^+\). 
Combining the two signs, we have
\[
  \sum_\pm\sum_{|\alpha|\le1}\sup_{x\in\mathcal C_b^\pm,\,s'\in\operatorname{supp}\chi_{\rm im}}
  |K_{\sigma,\alpha}^{\pm,b}(x,s')|\le \mathcal P(\bar\sigma)e^{-c_b\bar\sigma}.
\]
By superposition,
\[
  \partial_x^\alpha U_{\sigma,t}^\pm(x)=\int_{\operatorname{supp}\chi_{\rm im}}
  K_{\sigma,\alpha}^{\pm,b}(x,s')q_\sigma(s')\,ds' .
\]
Therefore
\[
  \sum_\pm\|U_{\sigma,t}^\pm\|_{H^1(\mathcal C_b^\pm)}
  \le \mathcal P(\bar\sigma)e^{-c_b\bar\sigma}\|q_\sigma\|_{Y_R},
\]
which proves \eqref{eq:QBim-separated-U-small}.
\end{proof}
Combining the above results, we obtain
\begin{lemma}\label{lem:QBim-family}
We have
\begin{equation}\label{eq:QBim-family}
  \left\|\{Q_\sigma\mathcal B_\sigma^{\rm im}h_\sigma\}_{\bar\sigma\ge\bar\sigma_0}\right\|_{\mathfrak Y^{a}}
  \le C\bar\sigma_0^{-1/2}\|\mathbf h\|_{\mathfrak Y^{a}} .
\end{equation}
\end{lemma}

\begin{proof}
The polynomial one-sided stability gives
\begin{equation}\label{eq:QBim-poly-Q}
  \|Q_\sigma r\|_{Y_R}
  \le \mathcal P(\bar\sigma)\|r\|_{Y_R},
  \qquad r\in Y_R .
\end{equation}
Substituting \(r=f_\sigma\) and using \eqref{eq:QBim-f-size} yields
\begin{equation}\label{eq:QBim-fixed-unweighted-proof}
  \|Q_\sigma f_\sigma\|_{Y_R}
  \le \mathcal P(\bar\sigma)\bar\sigma^{-1/2}\|q_\sigma\|_{Y_R}.
\end{equation}
By Lemma~\ref{lem:flat-symbol-ratio} and \eqref{eq:QBim-f-size}:
\begin{align}
  W_\sigma\|\chi_{\rm im}Q_{\sigma,t}f_\sigma\|_{Y_R}
  \le C W_\sigma\|f_\sigma\|_{Y_R}
  \le C\bar\sigma^{-1/2}W_\sigma\|q_\sigma\|_{Y_R} .
  \label{eq:QBim-flat-active}
\end{align}
Lemmas~\ref{lem:QBim-correction-separated} and \ref{lem:QBim-separated-fields}
give the estimate
\begin{equation}\label{eq:QBim-return-small}
  W_\sigma\|\chi_{\rm im}(Q_\sigma-Q_{\sigma,t})f_\sigma\|_{Y_R}
  \le \mathcal P(\bar\sigma) e^{-c_b\bar\sigma}
       W_\sigma\|q_\sigma\|_{Y_R} .
\end{equation}
Equations \eqref{eq:QBim-flat-active} and \eqref{eq:QBim-return-small} yield
\begin{equation}\label{eq:QBim-fixed-weighted-proof}
  W_\sigma\|\chi_{\rm im}Q_\sigma f_\sigma\|_{Y_R}
  \le C\left(\bar\sigma^{-1/2}
      +\mathcal P(\bar\sigma) e^{-c_b\bar\sigma}\right)
     W_\sigma\|q_\sigma\|_{Y_R} .
\end{equation}
Finally,
\begin{equation}\label{eq:QBim-X-use}
  \|q_\sigma\|_{Y_R}
  \le W_\sigma^{-1}\|\mathbf h\|_{\mathfrak Y^{a}},
  \qquad
  W_\sigma\|q_\sigma\|_{Y_R}
  \le \|\mathbf h\|_{\mathfrak Y^{a}} .
\end{equation}

Equations \eqref{eq:QBim-fixed-unweighted-proof} and \eqref{eq:QBim-X-use} give
\begin{equation}\label{eq:QBim-unweighted-family}
  \|Q_\sigma f_\sigma\|_{Y_R}
  \le \mathcal P(\bar\sigma)\bar\sigma^{-1/2}W_\sigma^{-1}
       \|\mathbf h\|_{\mathfrak Y^{a}} .
\end{equation}
Equations \eqref{eq:QBim-fixed-weighted-proof} and \eqref{eq:QBim-X-use} give
\begin{equation}\label{eq:QBim-weighted-family}
  W_\sigma\|\chi_{\rm im}Q_\sigma f_\sigma\|_{Y_R}
  \le C\left(\bar\sigma^{-1/2}
      +\mathcal P(\bar\sigma) e^{-c_b\bar\sigma}\right)
      \|\mathbf h\|_{\mathfrak Y^{a}} .
\end{equation}
As
\begin{equation}\label{eq:QBim-final-sup}
  \sup_{\bar\sigma\ge\bar\sigma_0}
  \left(\mathcal P(\bar\sigma)\bar\sigma^{-1/2}W_\sigma^{-1}
  +\bar\sigma^{-1/2}
  +\mathcal P(\bar\sigma) e^{-c_b\bar\sigma}\right)
  \le C\bar\sigma_0^{-1/2} ,
\end{equation}
taking the supremum over \(\bar\sigma\ge\bar\sigma_0\) in the two components of the \(\mathfrak Y^{a}\)-norm proves \eqref{eq:QBim-family}.
\end{proof}
Combining Lemmas~\ref{lem:Bim-family} and \ref{lem:QBim-family}, we justify  
\begin{corollary}\label{cor:Ainv-Bim-small}
We have
\begin{equation}\label{eq:Ainv-Bim-small}
  \left\|(\mathcal A_{\sigma,N}^0)^{-1}\mathcal B_{\sigma}^{\rm im}\right\|_{\mathfrak Y^{a}\to\mathfrak Y^{a}}
  \le C\bar\sigma_0^{-1/2} .
\end{equation}
\end{corollary}

\begin{proof}
For \(\mathbf h\in\mathfrak Y^{a}\), put
\begin{equation}\label{eq:Ainv-Bim-f}
  \mathbf f:=\mathcal B^{\rm im}\mathbf h .
\end{equation}
By Lemma~\ref{lem:exact-inverse-identity}, for each fixed \(\bar\sigma\),
\begin{equation}\label{eq:Ainv-Bim-identity}
  (\mathcal A_{\sigma,N}^0)^{-1}\mathcal B_\sigma^{\rm im}h_\sigma
  =\mathcal B_\sigma^{\rm im}h_\sigma-Q_\sigma\mathcal B_\sigma^{\rm im}h_\sigma .
\end{equation}
Lemma~\ref{lem:Bim-family} estimates the first term and Lemma~\ref{lem:QBim-family} estimates the second one.  Thus
\begin{equation}\label{eq:Ainv-Bim-proof}
  \|(\mathcal A_{\sigma,N}^0)^{-1}\mathcal B^{\rm im}\mathbf h\|_{\mathfrak Y^{a}}
  \le C\bar\sigma_0^{-1/2}\|\mathbf h\|_{\mathfrak Y^{a}} .
\end{equation}
\end{proof}

\subsection{\texorpdfstring{Well-posedness and convergence theory}{Well-posedness and convergence theory}}\label{subsec:neumann-series}
We are ready to establish the well-posedness and convergence theory for the computable PML-BIE \eqref{eq:computable-bie2}. 
\begin{lemma}\label{lem:source-defect}
There exists \(c_2>0\), independent of \(\bar\sigma_0\), such that, for any \(0<a<c_2\),
\begin{equation}\label{eq:source-defect}
  \left\|(\mathcal A_{\sigma,N}^0)^{-1}
  (\mathcal S_\sigma^c-\mathcal S_{\sigma,N}^0)g\right\|_{\mathfrak Y^{a}}
  \le \mathcal P(\bar\sigma_0)e^{-(c_2-a)\bar\sigma_0}\|g\|_{Y_R'} .
\end{equation}
\end{lemma}

\begin{proof}

Set
\[
  r_\sigma:=(\mathcal S_\sigma^c-\mathcal S_{\sigma,N}^0)g .
\]
As the support of \(g\) is away from $\Sigma_{\rm pml}$, the integrand in the
integral of $r_\sigma$ decays exponentially with exponent proportional to
$\bar{\sigma}$ according to Proposition~\ref{prop:reflected-kernel-estimate}.
Therefore, after choosing \(c_2>0\) smaller than the decay rate therein, we have
\begin{equation}\label{eq:source-defect-r}
  \|r_\sigma\|_{Y_R}+W_\sigma\|\chi_{\rm im}r_\sigma\|_{Y_R}
  \le \mathcal P(\bar\sigma)e^{-(c_2-a)\bar\sigma}\|g\|_{Y_R'} .
\end{equation}
Here the weight is harmless because \(W_\sigma=e^{a\bar\sigma}\) and we can take \(a<c_2\). On the other hand, the polynomial stability of $Q_\sigma$ directly implies
\begin{equation}\label{eq:source-defect-Qr}
  \|Q_\sigma r_\sigma\|_{Y_R}
  +W_\sigma\|\chi_{\rm im}Q_\sigma r_\sigma\|_{Y_R}
  \le \mathcal P(\bar\sigma)e^{-(c_2-a)\bar\sigma}\|g\|_{Y_R'} .
\end{equation}
Finally, \((\mathcal A_{\sigma,N}^0)^{-1}r_\sigma=r_\sigma-Q_\sigma r_\sigma\) by Lemma~\ref{lem:exact-inverse-identity}.  Combining the last two estimates and taking the supremum over \(\bar\sigma\ge\bar\sigma_0\) in the \(\mathfrak Y^{a}\)-norm proves \eqref{eq:source-defect}.
\end{proof}
Our main result is stated below.
\begin{theorem}\label{thm:computable-convergence}
Let \(\phi_\sigma^0\) solve the exact PML-Green BIE~\eqref{eq:exact-bie}.  Choose
$0<a<\min\{c_0,c_1,c_2\}$, where 
\(c_0=\alpha\) is the rate from Corollary~\ref{cor:exact-physical-family-bound}
with $\chi_{\rm tail}=\chi_{\rm im}$, \(c_1\) is the separated-remainder rate
from Lemma~\ref{lem:Ainv-E-small}, and \(c_2\) is the source-defect rate from
Lemma~\ref{lem:source-defect}.
Then the computable PML-BIE \eqref{eq:computable-bie2} is uniquely solvable in the weighted family class \(\mathfrak Y^{a}\); in particular, for each \(\bar\sigma\ge \bar\sigma_0\), it gives a density \(\phi_\sigma^c\in Y_R\).  Moreover, there exist constants \(c>0\), independent of \(\bar\sigma_0\), such that
\begin{equation}\label{eq:density-convergence}
  \|\boldsymbol\phi^c-\boldsymbol\phi^0\|_{\mathfrak Y^{a}}
  \le \mathcal P(\bar\sigma_0)e^{-c\bar\sigma_0}\|g\|_{Y_R'} .
\end{equation}
For any compact set \(K\Subset\Omega_p\),
\begin{equation}\label{eq:field-convergence}
  \|u-u_\sigma^c\|_{H^1(K)}
  \le \mathcal P(\bar\sigma_0)\left(e^{-c\bar\sigma_0}
    +e^{-Z(L_2+d_2)}\right)
  \|g\|_{Y_R'} .
\end{equation}
\end{theorem}

\begin{proof}
Set
\begin{equation}\label{eq:P-perturb-def}
  \mathcal P_\sigma:=\mathcal B_\sigma^{\rm im}+\mathcal E_\sigma
  =\mathcal A_{\sigma,N}^0-\mathcal A_\sigma^c .
\end{equation}
We search the computable density in the form \(\phi_\sigma^c=\phi_\sigma^0+e_\sigma\).  As \(\mathcal A_{\sigma,N}^0\phi_\sigma^0=\mathcal S_{\sigma,N}^0g\), equation \eqref{eq:computable-bie2} is equivalent to
\begin{equation}\label{eq:error-fixed-point}
  e_\sigma
  =(\mathcal A_{\sigma,N}^0)^{-1}\mathcal P_\sigma e_\sigma
   +(\mathcal A_{\sigma,N}^0)^{-1}
     \left\{(\mathcal S_\sigma^c-\mathcal S_{\sigma,N}^0)g
     +\mathcal P_\sigma\phi_\sigma^0\right\} .
\end{equation}
By Corollary~\ref{cor:Ainv-Bim-small} and Lemma~\ref{lem:Ainv-E-small},
\begin{equation}\label{eq:P-small}
  \left\|(\mathcal A_{\sigma,N}^0)^{-1}\mathcal P_\sigma\right\|_{\mathfrak Y^{a}\to\mathfrak Y^{a}}
  \le C\bar\sigma_0^{-1/2}+C\mathcal P(\bar\sigma_0)e^{-(c_1-a)\bar\sigma_0}<\frac12
\end{equation}
for \(\bar\sigma_0\) large.  Hence \(I-(\mathcal A_{\sigma,N}^0)^{-1}\mathcal P_\sigma\) is invertible on \(\mathfrak Y^{a}\) by the Neumann series.  This proves the existence and uniqueness of \(e_\sigma\), and therefore of the solution \(\phi_\sigma^c=\phi_\sigma^0+e_\sigma\), provided that the inhomogeneous term in \eqref{eq:error-fixed-point} belongs to \(\mathfrak Y^{a}\). The two terms in \(\mathcal P_\sigma\phi_\sigma^0\) are estimated directly to justify this fact; note that we cannot use \eqref{eq:P-small} for the estimations.  For the separated remainder $\mathcal E_\sigma$, Lemma~\ref{lem:Ainv-E-small} already gives
\begin{equation}\label{eq:Ainv-E-phi0-small}
  \|(\mathcal A_{\sigma,N}^0)^{-1}\mathcal E_\sigma\phi_\sigma^0\|_{\mathfrak Y^{a}}
  \le \mathcal P(\bar\sigma_0) e^{-(c_1-a)\bar\sigma_0}\|g\|_{Y_R'} .
\end{equation}
For the nearest-image term ${\cal B}_{\sigma}^{\rm im}$, as \(\mathcal B_\sigma^{\rm im}\) only sees \(\chi_{\rm im}\phi_\sigma^0\), \eqref{eq:exact-physical-family-tail-small}, Lemma~\ref{lem:Bim-unweighted}, and Lemma~\ref{lem:QBim-family} give
\begin{equation}\label{eq:Ainv-Bim-phi0-small}
  \|(\mathcal A_{\sigma,N}^0)^{-1}\mathcal B_{\sigma}^{\rm im}\phi_\sigma^0\|_{\mathfrak Y^{a}}
  \le \mathcal P(\bar\sigma_0)e^{-(c_0-a)\bar\sigma_0}\|g\|_{Y_R'} .
\end{equation}
Therefore, by applying Lemma~\ref{lem:source-defect}, the inhomogeneous term in \eqref{eq:error-fixed-point} is bounded by
  $C\mathcal P(\bar\sigma_0)e^{-c\bar\sigma_0}\|g\|_{Y_R'}$ by choosing $0<c<\min_j{c_j-a}$.
The compact-field estimate \eqref{eq:field-convergence} follows directly from the exact PML-PDE convergence estimate \eqref{eq:pde-conv}.
\end{proof}

\section{\texorpdfstring{The complex-scaled BIE}{The complex-scaled BIE}}\label{sec:cs-bie}

Theorem~\ref{thm:computable-convergence} proves the convergence of the computable PML-BIE \eqref{eq:computable-bie2}.  This finite equation can be viewed as a hard truncation of an untruncated complex-scaled BIE on the whole impedance boundary \(\Sigma\).

Let \(u\) be the outgoing solution of \eqref{eq:phys-pde}--\eqref{eq:phys-src}.  In the PML region define the formal complex-scaled extension
  $u_\sigma^\infty(x):=u(F_\sigma(x))$,
by the outgoing analytic continuation \cite[Sec. 4]{JiangLi2020}. In the physical region $\Omega_p$, where \(F_\sigma(x)=x\),
  $u_\sigma^\infty(x)=u(x)$.
Let \(\mathcal S_\sigma^\infty\) and \(\mathcal K_\sigma^\infty\) be the single- and double-layer boundary integral operators defined the same as ${\cal S}_\sigma^c$ and ${\cal K}_\sigma^c$ but with $\Gamma_R$ replaced by the whole impedance boundary \(\Sigma\).
Set
\begin{equation}\label{eq:Ainfty}
  \mathcal A_\sigma^\infty:=\frac12I+\mathcal K_\sigma^\infty+\mathcal S_\sigma^\infty M_\beta .
\end{equation}
The corresponding untruncated complex-scaled BIE is
\begin{equation}\label{eq:complex-scaled-bie}
  \mathcal A_\sigma^\infty\phi_\sigma^\infty=\mathcal S_\sigma^\infty g
  \qquad \text{on }\Sigma .
\end{equation}
It is clear that $\phi_\sigma^\infty =  u_\sigma^\infty|_{\Sigma}$ is a solution so that its physical part is the same as $u$. 

Indeed, the computable PML-BIE \eqref{eq:computable-bie2} is the hard truncation of \eqref{eq:complex-scaled-bie}: the infinite boundary \(\Sigma\) is replaced by the retained arc \(\Gamma_R\), and the whole-boundary operators \(\mathcal S_\sigma^\infty,\mathcal K_\sigma^\infty\) are replaced by their finite-arc counterparts \(\mathcal S_\sigma^c,\mathcal K_\sigma^c\).  Thus Theorem~\ref{thm:computable-convergence} directly gives the exponential convergence of $\phi_\sigma^c$ to the physical solution $u$, or the complex-scaled solution $\phi_\sigma^\infty$ in the physical region of $\Gamma_R$. The present analysis does not require a well-posedness theory for \eqref{eq:complex-scaled-bie}, and a practical numerical method indeed does not solve \eqref{eq:complex-scaled-bie}; this paper avoids the challenging analysis of integral operators on the unbounded curve \(\Sigma\).  For the convergence of the practical finite PML-BIE, such an untruncated problem is not needed: it serves only to interpret the stretched free-space kernel used in the computable equation.

\section{Conclusion}\label{sec:conclusion}

We have established a convergence theory for the PML-BIE method for acoustic scattering by an impedance half-space.  The framework proceeds in four steps.  First, one proves the well-posedness and exponential convergence of the underlying PML truncated PDE.  Second, one constructs an exact Neumann PML Green function and the associated exact PML-Green BIE; this BIE must be equivalent to the original PML PDE, so that the PDE stability can be transferred to the boundary integral formulation.  Third, one compares this exact PML Green function with the stretched free-space Green function.  Apart from the stretched free-space term, the difference consists of the nearest Neumann reflected image and well-separated reflected terms: the well-separated terms are exponentially small, while the nearest-image contribution is shown to be algebraically small in a weighted space.  Finally, a Neumann-series argument transfers the stability and convergence from the exact PML-Green BIE to the computable stretched-kernel PML-BIE.  We believe that this provides a general framework that can be adapted to analyze more complicated scattering structures.

Several extensions are natural.  For scalar transmission and layered-medium scattering problems, the exact-kernel strategy should lead to Calder\'on or M\"uller-type systems for Cauchy data on the interface; the related PML-BIE methodology has already appeared in acoustic and electromagnetic layered media \cite{LuLuQian2018,LuXuYinZhang2023,BaoLuYinZhang2024,WangLu2026}.  For water-wave scattering, complex-scaled boundary integral formulations such as \cite{BonnetBenDhiaFariaPerez2024} provide a natural target for the same framework; we note however that a significant difference is that the stretched Laplacian kernel therein does not decay exponentially.  Maxwell transmission problems and step-like or multi-layered interfaces, including settings related to \cite{YuHuLuRathsfeld2022,Lu2021,LuZhengZhu2025,LuLi2025}, are also natural directions.

\end{document}